\documentclass[11pt,a4paper]{article}

\usepackage{titlesec}
\usepackage{fancyhdr}
\usepackage{a4wide}
\usepackage{graphicx}
\usepackage{float}
\usepackage{amssymb}
\usepackage{amsmath}
\usepackage{amsfonts}
\usepackage{dsfont}
\usepackage{amsthm}
\usepackage{color}
\usepackage{mathrsfs}
\usepackage{array}
\usepackage{eucal}
\usepackage{tikz}
\usepackage[T1]{fontenc}
\usepackage{inputenc}
\usepackage[english]{babel}
\usepackage{lmodern}
\usepackage{hyperref}
\usepackage{geometry}
\usepackage{changepage}
\usepackage{bm}
\changepage{0pt}{}{}{}{}{0pt}{}{0pt}{6pt}
\usepackage[numbers]{natbib}
\usepackage{bbm}

\usepackage{hyperref}
\hypersetup{
pdfpagemode=none,
pdftoolbar=true,        
pdfmenubar=true,        
pdffitwindow=false,     
pdfstartview={Fit},    
pdftitle={Spectrum of the Dirichlet Laplacian in a thin cubic lattice},    
pdfauthor={L. Chesnel, S.A. Nazarov},     
pdfsubject={},  
pdfcreator={L. Chesnel, S.A. Nazarov},   
pdfproducer={}, 
pdfkeywords={}, 
pdfnewwindow=true,      
colorlinks=true,       
linkcolor=magenta,          
citecolor=red,        
filecolor=cyan,      
urlcolor=blue           
}

\newcommand{\dsp}{\displaystyle}

\newcommand{\eps}{\varepsilon}
\newcommand{\om}{\omega}
\newcommand{\Om}{\Omega}
\newcommand{\mrm}[1]{\mathrm{#1}}

\newcommand{\Cplx}{\mathbb{C}}
\newcommand{\N}{\mathbb{N}}
\newcommand{\R}{\mathbb{R}}
\newcommand{\Z}{\mathbb{Z}}

\newcommand{\mL}{\mrm{L}}
\newcommand{\mH}{\mrm{H}}

\renewcommand{\ker}{\mrm{ker}}

\renewcommand{\dim}{\mrm{dim}}

\newtheorem{theorem}{Theorem}[section]
\newtheorem{lemma}[theorem]{Lemma}
\newtheorem{remark}[theorem]{Remark}

\newtheorem{proposition}[theorem]{Proposition}

\begin{document}

~\vspace{-0.4cm}
\begin{center}
{\sc \bf\huge  
Spectrum of the Dirichlet Laplacian\\[6pt] in a thin cubic lattice}
\end{center}

\begin{center}
\textsc{Lucas Chesnel}$^1$, \textsc{Sergei A. Nazarov}$^{2}$\\[16pt]
\begin{minipage}{0.95\textwidth}
{\small
$^1$ Inria, Ensta Paris, Institut Polytechnique de Paris, 828 Boulevard des Mar\'echaux, 91762 Palaiseau, France;\\
$^2$ Institute of Problems of Mechanical Engineering RAS, V.O., Bolshoi pr., 61, St. Petersburg, 199178, Russia;\\
E-mails: \texttt{lucas.chesnel@inria.fr}, \texttt{srgnazarov@yahoo.co.uk} \\[-14pt]
\begin{center}
(\today)
\end{center}
}
\end{minipage}
\end{center}
\vspace{0.2cm}

\noindent\textbf{Abstract.} We give a description of the lower part of the spectrum of the Dirichlet Laplacian in an unbounded 3D periodic lattice made of thin bars (of width $\eps\ll1$) which have a square cross section. This spectrum coincides with the union of segments which all go to $+\infty$ as $\eps$ tends to zero due to the Dirichlet boundary condition. We show that the first spectral segment is extremely tight, of length $O(e^{-\delta/\eps})$, $\delta>0$, while the length of the next spectral segments is $O(\eps)$. To establish these results, we need to study in detail the properties of the Dirichlet Laplacian $A^{\Om}$ in the geometry $\Om$ obtained by zooming at the junction regions of the initial periodic lattice. This problem has its own interest and playing with symmetries together with max-min arguments as well as a well-chosen Friedrichs inequality, we prove that $A^{\Om}$ has a unique eigenvalue in its discrete spectrum, which generates the first spectral segment. Additionally we show that there is no threshold resonance for $A^{\Om}$, that is no non trivial bounded solution at the threshold frequency for $A^{\Om}$. This implies that the correct 1D model of the lattice for the next spectral segments is a graph with Dirichlet conditions at the vertices. We also present numerics to complement the analysis.\\

\noindent\textbf{Key words.} Quantum waveguide, thin periodic lattice, threshold resonance, trapped waves. 

\section{Introduction}

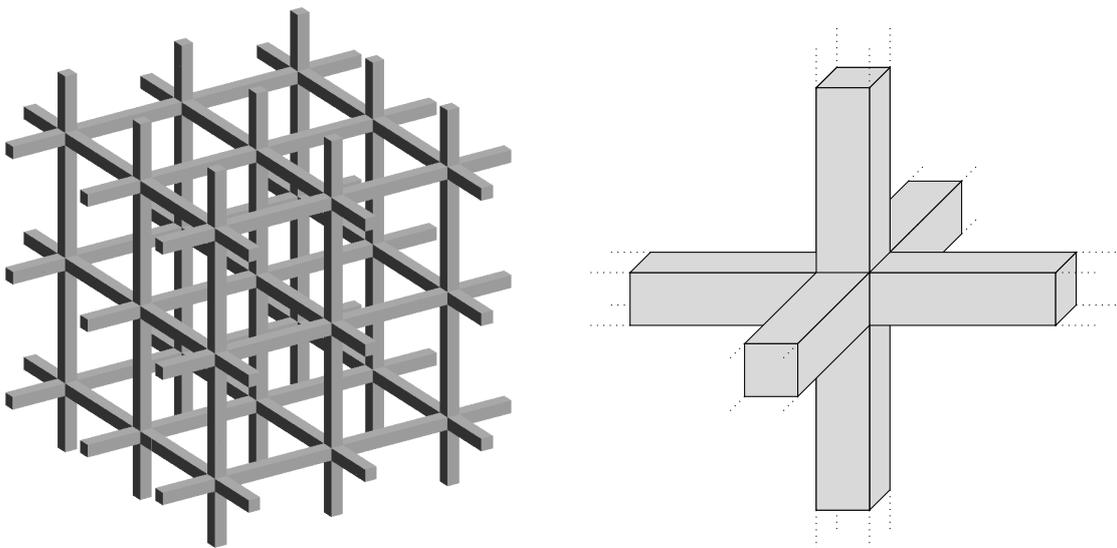
\begin{figure}[!ht]
\centering
\includegraphics[width=7cm]{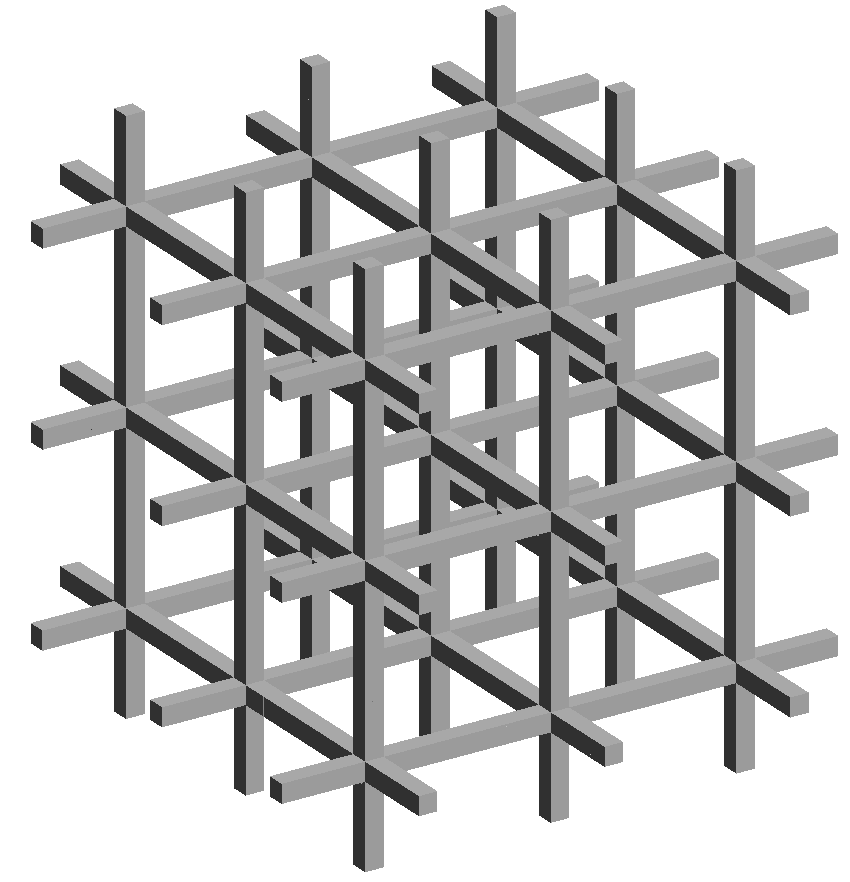}
\qquad\begin{tikzpicture}[scale=0.7]
\pgfmathsetmacro{\cubexa}{8}
\pgfmathsetmacro{\cubeya}{1}
\pgfmathsetmacro{\cubeza}{1}

\draw[black,dotted] (-0.5,-\cubexa/2,-0.5) -- ++(0,-0.8,0);
\draw[black,dotted] (-\cubexa/2,-0.5,-0.5) -- ++(-0.8,0,0);

\draw[black,fill=gray!30] (0.5,0.5,-0.5) -- ++(0,0,-\cubexa/2+0.5) -- ++(-\cubeya,0,0) -- ++(0,0,\cubexa/2-0.5) -- cycle;
\draw[black,fill=gray!30] (0.5,0.5,-0.5) -- ++(0,-\cubeza,0) -- ++(-\cubeya,0,0) -- ++(0,\cubeza,0) -- cycle;
\draw[black,fill=gray!30] (0.5,0.5,-0.5) -- ++(0,0,-\cubexa/2+0.5) -- ++(0,-\cubeza,0) -- ++(0,0,\cubexa/2-0.5) -- cycle;

\draw[black,fill=gray!30] (0.5,\cubexa/2,0.5) -- ++(0,-\cubexa,0) -- ++(0,0,-\cubeya) -- ++(0,\cubexa,0) -- cycle;
\draw[black,fill=gray!30] (0.5,\cubexa/2,0.5) -- ++(-\cubeza,0,0) -- ++(0,0,-\cubeya) -- ++(\cubeza,0,0) -- cycle;
\draw[black,fill=gray!30] (0.5,\cubexa/2,0.5) -- ++(0,-\cubexa,0) -- ++(-\cubeza,0,0) -- ++(0,\cubexa,0) -- cycle;

\draw[black,fill=gray!30] (\cubexa/2,0.5,0.5) -- ++(-\cubexa,0,0) -- ++(0,-\cubeya,0) -- ++(\cubexa,0,0) -- cycle;
\draw[black,fill=gray!30] (\cubexa/2,0.5,0.5) -- ++(0,0,-\cubeza) -- ++(0,-\cubeya,0) -- ++(0,0,\cubeza) -- cycle;
\draw[black,fill=gray!30] (\cubexa/2,0.5,0.5) -- ++(-\cubexa,0,0) -- ++(0,0,-\cubeza) -- ++(\cubexa,0,0) -- cycle;

\draw[black,fill=gray!30] (0.5,\cubexa/2,0.5) -- ++(0,-\cubexa/2+0.5,0) -- ++(0,0,-\cubeya) -- ++(0,\cubexa/2-0.5,0) -- cycle;
\draw[black,fill=gray!30] (0.5,\cubexa/2,0.5) -- ++(-\cubeza,0,0) -- ++(0,0,-\cubeya) -- ++(\cubeza,0,0) -- cycle;
\draw[black,fill=gray!30] (0.5,\cubexa/2,0.5) -- ++(0,-\cubexa,0) -- ++(-\cubeza,0,0) -- ++(0,\cubexa,0) -- cycle;

\draw[black,fill=gray!30] (0.5,0.5,\cubexa/2) -- ++(0,0,-\cubexa/2+0.5) -- ++(-\cubeya,0,0) -- ++(0,0,\cubexa/2-0.5) -- cycle;
\draw[black,fill=gray!30] (0.5,0.5,\cubexa/2) -- ++(0,-\cubeza,0) -- ++(-\cubeya,0,0) -- ++(0,\cubeza,0) -- cycle;
\draw[black,fill=gray!30] (0.5,0.5,\cubexa/2) -- ++(0,0,-\cubexa/2+0.5) -- ++(0,-\cubeza,0) -- ++(0,0,\cubexa/2-0.5) -- cycle;
\draw[black,dotted] (\cubexa/2,0.5,0.5) -- ++(0.8,0,0);
\draw[black,dotted] (\cubexa/2,-0.5,0.5) -- ++(0.8,0,0);
\draw[black,dotted] (\cubexa/2,0.5,-0.5) -- ++(0.8,0,0);
\draw[black,dotted] (\cubexa/2,-0.5,-0.5) -- ++(0.8,0,0);
\draw[black,dotted] (-\cubexa/2,0.5,0.5) -- ++(-0.8,0,0);
\draw[black,dotted] (-\cubexa/2,-0.5,0.5) -- ++(-0.8,0,0);
\draw[black,dotted] (-\cubexa/2,0.5,-0.5) -- ++(-0.8,0,0);

\draw[black,dotted] (0.5,\cubexa/2,0.5) -- ++(0,0.8,0);
\draw[black,dotted] (0.5,\cubexa/2,-0.5) -- ++(0,0.8,0);
\draw[black,dotted] (-0.5,\cubexa/2,0.5) -- ++(0,0.8,0);
\draw[black,dotted] (-0.5,\cubexa/2,-0.5) -- ++(0,0.8,0);
\draw[black,dotted] (0.5,-\cubexa/2,0.5) -- ++(0,-0.8,0);
\draw[black,dotted] (0.5,-\cubexa/2,-0.5) -- ++(0,-0.8,0);
\draw[black,dotted] (-0.5,-\cubexa/2,0.5) -- ++(0,-0.8,0);

\draw[black,dotted] (0.5,0.5,\cubexa/2) -- ++(0,0,0.8);
\draw[black,dotted] (-0.5,0.5,\cubexa/2) -- ++(0,0,0.8);
\draw[black,dotted] (0.5,-0.5,\cubexa/2) -- ++(0,0,0.8);
\draw[black,dotted] (-0.5,-0.5,\cubexa/2) -- ++(0,0,0.8);
\draw[black,dotted] (0.5,0.5,-\cubexa/2) -- ++(0,0,-0.8);
\draw[black,dotted] (-0.5,0.5,-\cubexa/2) -- ++(0,0,-0.8);
\draw[black,dotted] (0.5,-0.5,-\cubexa/2) -- ++(0,0,-0.8);
\draw[black,dotted] (-0.5,-0.5,-\cubexa/2) -- ++(0,0,-0.8);
\end{tikzpicture}
\caption{Left: unbounded periodic cubic lattice made of thin bars of width $\eps$ with square cross-sections. Right: geometry $\Om$ of the near field, aka boundary layer, problem at the junction regions. \label{PeriodicLattice}}
\end{figure}

With the profusion of works related to graphene in physics, an important effort has been made in the mathematical community to understand the asymptotic behaviour of the spectrum of the Dirichlet operator in quantum waveguides made of thin ligaments, of characteristic width $\eps\ll1$, forming unbounded periodic lattices. Various geometries have been considered and we refer the reader to \cite{Kuch02} for a review article. 
To address such problems, the general approach can be summarized as follows. Using the Floquet-Bloch-Gelfand theory \cite{Gelf50,Kuch82,Skri87,Kuch93}, one shows that the spectrum of the operator in the periodic domain has a band-gap structure, the bands being generated by the eigenvalues of a spectral problem set on the periodicity cell with quasi-periodic boundary conditions involving the Floquet-Bloch parameter. The first step consists in applying techniques of dimension reduction to derive a 1D model for this spectral problem on the periodicity cell. Then one studies precisely this 1D model depending on the Floquet-Bloch parameter to get information on the spectral bands.\\
\newline
Let us mention that in the past, slapdash and casual conclusions have been made concerning the 1D model problem. This model consists of ordinary differential equations on the ligaments obtained when taking $\eps\to0$ supplemented by transmission conditions at the nodes of the graph. Certain authors have inappropriately applied L. Pauling's model \cite{Paul36} and imposed Kirchoff transmission conditions at the nodes. These Kirchoff conditions boil down to impose continuity of the field and zero outgoing flux (the sum of the derivatives of the field along the outgoing directions at the node vanishes). This is correct for the Laplacian with Neumann boundary conditions (BC) and has been rigorously justified in \cite{KuZe03,ExPo05,Post12}. However it has been shown by D. Grieser in \cite{Grie08} (see also \cite{MoVa07}) that for the Dirichlet problem, in general the right conditions to impose at the nodes are Dirichlet ones. More precisely, it has been proved in \cite{Grie08} that the transmission conditions to impose depend on the existence or absence of so-called threshold resonances for the near field operator defined as the Laplacian in the geometry obtained when zooming at the junction regions (denoted $\Om$ in the sequel, see Figure \ref{PeriodicLattice} right). We say that there is a threshold resonance if there is a non zero bounded function which solves the homogeneous problem at the frequency coinciding with the bottom of the essential spectrum (the threshold) of the Laplace operator. For the Neumann problem, the threshold is $\Lambda_\dagger=0$ and  there is a threshold resonance because the constants solve $-\Delta u=0$ in $\Om$ + Neumann BC. Due to this property, one must impose Kirchoff conditions at the nodes. For the Dirichlet problem, the continuous spectrum starts at a positive threshold $\Lambda_\dagger>0$ and in general the only solution to the problem 
\begin{equation}\label{PbThreshold}
\begin{array}{|rcll}
-\Delta u&=&\Lambda_\dagger u &\mbox{ in } \Om\\[3pt]
u&=&0&\mbox{ on } \partial\Om
\end{array}
\end{equation}
which remains bounded at infinity is zero, i.e. there is no threshold resonance. Because of this feature, one should impose Dirichlet condition at the nodes of the 1D model (see (\ref{farfield_BC}) for the precise moment where this pops up in the analysis below).

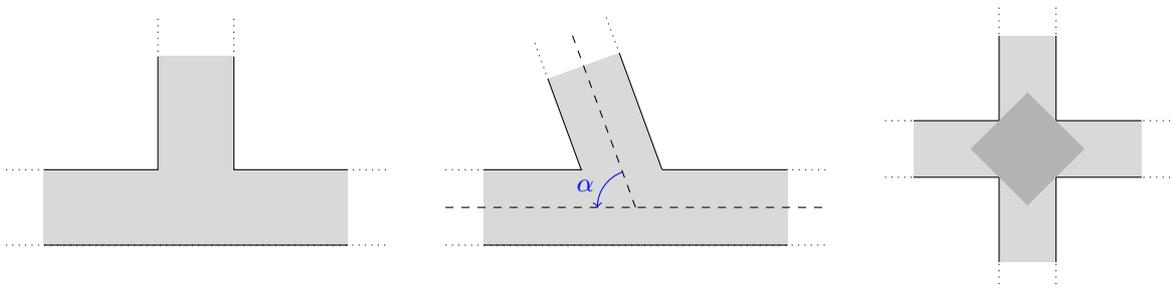
\begin{figure}[!ht]
\centering
\begin{tikzpicture}[scale=1]
\draw[fill=gray!30,draw=none](-2,-1/2) rectangle (2,1/2);
\draw[fill=gray!30,draw=none](-1/2,-1/2) rectangle (1/2,2);
\draw[black] (-2,-1/2)--(2,-1/2);
\draw[black] (-2,1/2)--(-1/2,1/2)--(-1/2,2);
\draw[black] (2,1/2)--(1/2,1/2)--(1/2,2);
\draw[black,dotted] (-2.5,-1/2)--(2.5,-1/2);
\draw[black,dotted] (-2.5,1/2)--(-2,1/2);
\draw[black,dotted] (2.5,1/2)--(2,1/2);
\draw[black,dotted] (-1/2,2.5)--(-1/2,2);
\draw[black,dotted] (1/2,2.5)--(1/2,2);
\end{tikzpicture}\qquad\begin{tikzpicture}[scale=1]
\begin{scope}[rotate=20]
\draw[fill=gray!30,draw=none](-1/2,0) rectangle (1/2,2);
\draw[black] (-1/2,1/2)--(-1/2,2);
\draw[black] (1/2,0)--(1/2,2);
\draw[black,dotted] (-1/2,2.5)--(-1/2,2);
\draw[black,dotted] (1/2,2.5)--(1/2,2);
\end{scope}
\draw[fill=gray!30,draw=none](-2,-1/2) rectangle (2,1/2);
\begin{scope}[rotate=20]
\draw[black,dashed] (0,0)--(0,2.5);
\end{scope}
\draw[black,dashed] (-2.5,0)--(2.5,0);
\draw[black] (-2,-1/2)--(2,-1/2);
\draw[black] (-2,1/2)--(-0.7,1/2);
\draw[black] (2,1/2)--(0.35,1/2);
\draw[black,dotted] (-2.5,-1/2)--(2.5,-1/2);
\draw[black,dotted] (-2.5,1/2)--(-2,1/2);
\draw[black,dotted] (2.5,1/2)--(2,1/2);
\pgfmathparse{.5*cos(110)}\let\x\pgfmathresult
\pgfmathparse{.5*sin(110)}\let\y\pgfmathresult
\draw[blue,->] (\x,\y) arc (110:180:.5cm);
\node[left,blue] at (145:.5) {\small $\alpha$};
\end{tikzpicture}\qquad\raisebox{-0.6cm}{\begin{tikzpicture}[scale=0.75]
\draw[fill=gray!30,draw=none](-2,-1/2) rectangle (2,1/2);
\draw[fill=gray!30,draw=none](-1/2,-2) rectangle (1/2,2);
\draw[black] (-2,-1/2)--(-1/2,-1/2)--(-1/2,-2);
\draw[black] (2,-1/2)--(1/2,-1/2)--(1/2,-2);
\draw[black] (-2,1/2)--(-1/2,1/2)--(-1/2,2);
\draw[black] (2,1/2)--(1/2,1/2)--(1/2,2);
\draw[black,dotted] (2.5,-1/2)--(2,-1/2);
\draw[black,dotted] (-2.5,-1/2)--(-2,-1/2);
\draw[black,dotted] (-2.5,1/2)--(-2,1/2);
\draw[black,dotted] (2.5,1/2)--(2,1/2);
\draw[black,dotted] (-1/2,2.5)--(-1/2,2);
\draw[black,dotted] (1/2,2.5)--(1/2,2);
\draw[black,dotted] (-1/2,-2.5)--(-1/2,-2);
\draw[black,dotted] (1/2,-2.5)--(1/2,-2);
\begin{scope}[rotate=45]
\draw[fill=gray!60,draw=none](-0.707,-0.707) rectangle (0.707,0.707);
\end{scope}
\end{tikzpicture}}
\caption{Examples of T- and X-shaped 2D geometries.\label{2Dexamples}} 
\end{figure}

\noindent From there, authors have worked to establish rigorous results showing the absence of non zero bounded solution to (\ref{PbThreshold}). First, different planar
quantum waveguides made of T-, X- and Y-shaped junctions of thin ligaments have been considered in \cite{Naza10a,Naza14,Naza14a,NaRU14,NaRU15,Naza17,Pank17}. In these articles, additionally it has been proved that the near field operator in $\Om$, depending on the considered geometry, may have discrete spectrum (one or several eigenvalue below the continuous spectrum). When the latter exists, the low-frequency range of the spectrum in the periodic domain is not described by the above mentioned 1D model with Dirichlet conditions at the nodes. Instead, the first spectral bands, their number being equal to the multiplicity of the discrete spectrum, are generated by functions which are localized at the junctions regions. Let us mention that for certain exceptional geometries, for example for a sequence of angles $\alpha$ in the central domain of Figure \ref{2Dexamples}, one may have non zero solutions to (\ref{PbThreshold}) which remain bounded at infinity, i.e. existence of threshold resonance. In these situations, at least if the dimension of the space of bounded solutions to (\ref{PbThreshold}) is one, the good 1D model describing the spectral bands which are not associated with the discrete spectrum in $\Om$, has certain Kirchoff transmission conditions at the nodes which depend on the geometry. We emphasize that this leads to very different spectra for the operator in the periodic medium. More precisely, when there is no threshold resonance, the bands of the spectrum in the periodic material become very small as $\eps\to0$ and the spectral gaps enlarge. In other words, the spectrum becomes rather sparse, for most of the spectral parameters waves cannot propagate and the limit 1D ligaments are somehow disconnected. This is what we will obtain below in our configuration. On the other hand, when Problem (\ref{PbThreshold}) admits a space of dimension one of bounded solutions, the spectral bands in the periodic material are much larger.\\
\newline
Afterwards, 3D geometries were considered in \cite{BaMaNa17} (see also the corresponding note \cite{BaMaNa15}). In these articles, the authors consider quantum waveguides for which the near field domain $\Om$ is a cruciform junction of two cylinders whose cross section coincides with the unit disc. The passage from the planar case to the spatial case requires the non obvious adaptation of the methods. In particular the characterization of the discrete spectrum of the near field operator and the proof of absence of threshold resonances are much more involved. 
In the present work, we study an even more intricate 3D geometry for which the near field geometry is the union of three waveguides. Rather precise Friedrichs estimates are required to prove that the discrete spectrum of the near field operator contains exactly one eigenvalue and to show that at the threshold, zero is the only bounded solution. This work also complements the study of \cite{BaMaNa17} thanks to the numerical experiments. Let us mention that the cruciform junction of two 3D cylinders with square cross section reduces to a 2D problem in a X-shaped geometry (see again Figure \ref{2Dexamples} right) and due to factoring out, is of no interest.\\
\newline
The outline is as follows. First we describe the problem, introduce the notation and present the main results. Then we study the discrete spectrum of the near field operator. In section \ref{sectionAbsenceofBoundedSol}, we demonstrate the absence of threshold resonance for the near field operator. Section \ref{SectionModels} is dedicated to the analysis of the main theorem of the article (Theorem \ref{MainThmPerio}) with the derivation of asymptotic models for the spectral bands in the original periodic domain. Finally we show some numerics to complement the results and conclude with some appendix containing the proof of two lemmas needed in sections \ref{Section_DiscreteSpectrum}, \ref{sectionAbsenceofBoundedSol}. 

\section{Notation and main results}

For $j=1,2,3$, introduce the cylinder with square cross section
\begin{equation}\label{defBars}
L_j  :=  \{x:=(x_1,x_2,x_2)\in\R^3\,|\,|x_k|<1/2\mbox{ for }k\in\{1,2,3\}\setminus\{j\}\}
\end{equation}
and for $\eps>0$ small, $m,n\in\Z:=\{0,\pm1,\pm2,\dots\}$, set
\[
\begin{array}{rcl}
L^{mn\eps}_1 & := & \{x\in\R^3\,|\,|x_2-m|<\eps/2,\,|x_3-n|<\eps/2\}\\[2pt]
L^{mn\eps}_2 & := & \{x\in\R^3\,|\,|x_1-m|<\eps/2,\,|x_3-n|<\eps/2\}\\[2pt]
L^{mn\eps}_3 & := & \{x\in\R^3\,|\,|x_1-m|<\eps/2,\,|x_2-n|<\eps/2\}.
\end{array}
\]
Finally define the unbounded periodic domain 
\[
\Pi^\eps:=\bigcup_{m,n\in\Z}L^{mn\eps}_1\cup L^{mn\eps}_2\cup L^{mn\eps}_3
\]
(see Figure \ref{PeriodicLattice} left). Consider the Dirichlet spectral problem for the Laplace operator 
\begin{equation}\label{PbSpectral}
\begin{array}{|rcll}
-\Delta u^\eps&=&\lambda^\eps\,u^\eps &\mbox{ in }\Pi^\eps\\[3pt]
u^\eps&=&0&\mbox{ on } \partial\Pi^\eps.
\end{array}
\end{equation}
The variational form associated with this problem writes
\begin{equation}\label{variationPb}
\int_{\Pi^\eps}\nabla u^\eps\cdot\nabla v^\eps\,dx=\lambda^\eps\int_{\Pi^\eps} u^\eps v^\eps\,dx,\qquad\forall v^\eps\in\mH^1_0(\Pi^\eps).
\end{equation}
Here $\mH^1_0(\Pi^\eps)$ stands for the Sobolev space of functions of $\mH^1(\Pi^\eps)$ which vanish on the boundary $\partial\Pi^\eps$. Classically (see e.g \cite[\S10.1]{BiSo87}), the variational problem (\ref{variationPb}) gives rise to an unbounded, positive definite, selfadjoint operator $A^\eps$ in the Hilbert space $\mL^2(\Pi^\eps)$, with domain $\mathcal{D}(A^\eps)\subset\mH^1_0(\Pi^\eps)$. Note that this operator is sometimes called the quantum graph Laplacian \cite{BeKu13,Post12}. Since $\Pi^\eps$ is unbounded, the embedding $\mH^1_0(\Pi^\eps)\subset \mL^2(\Pi^\eps)$ is not compact and $A^\eps$ has a non empty essential component $\sigma_e(A^\eps)$ (\cite[Thm. 10.1.5]{BiSo87}). Actually, due to the periodicity, we have $\sigma_e(A^\eps)=\sigma(A^\eps)$. The Gelfand transform (see the surveys \cite{Kuch82,Naza99a} and books \cite{Skri87,Kuch93}) 
\[
u^{\eps}(x)\mapsto U^\eps(x,\eta)=\cfrac{1}{(2\pi)^{3/2}}\sum_{\iota\in\Z^3} e^{i\eta\cdot \iota}u^{\eps}(x+\iota),\quad\eta=(\eta_1,\eta_2,\eta_3),
\]
changes Problem (\ref{PbSpectral}) into the following spectral problem with quasi-periodic at the faces located at $x_1=\pm1/2$, $x_2=\pm1/2$, $x_3=\pm1/2$,
\begin{equation}\label{PbSpectralCell}
\begin{array}{|rcll}
-\Delta U^\eps(x,\eta)&=&\Lambda^\eps\,U^\eps(x,\eta) & x\in\om^\eps\\[3pt]
U^\eps(x,\eta)&=&0&x\in\partial\om^\eps\cap\partial\Pi^\eps\\[3pt]
U^\eps(-1/2,x_2,x_3,\eta) & = & e^{i\eta_1}U^\eps(+1/2,x_2,x_3,\eta) & (x_2,x_3)\in(-\eps/2;\eps/2)^2\\[3pt]
\partial_{x_1}U^\eps(-1/2,x_2,x_3,\eta) & = & e^{i\eta_1}\partial_{x_1}U^\eps(+1/2,x_2,x_3,\eta) & (x_2,x_3)\in(-\eps/2;\eps/2)^2\\[3pt]
U^\eps(x_1,-1/2,x_3,\eta) & = & e^{i\eta_2}U^\eps(x_1,+1/2,x_3,\eta) & (x_1,x_3)\in(-\eps/2;\eps/2)^2\\[3pt]
\partial_{x_2}U^\eps(x_1,-1/2,x_3,\eta) & = & e^{i\eta_2}\partial_{x_2}U^\eps(x_1,+1/2,x_3,\eta) & (x_1,x_3)\in(-\eps/2;\eps/2)^2\\[3pt]
U^\eps(x_1,x_2,-1/2,\eta) & = & e^{i\eta_3}U^\eps(x_1,x_2,+1/2,\eta) & (x_1,x_2)\in(-\eps/2;\eps/2)^2\\[3pt]
\partial_{x_3}U^\eps(x_1,x_2,-1/2,\eta) & = & e^{i\eta_3}\partial_{x_3}U^\eps(x_1,x_2,+1/2,\eta) & (x_1,x_2)\in(-\eps/2;\eps/2)^2,
\end{array}
\end{equation}
set in the periodicity cell 
\[
\om^\eps:=\om^\eps_1\cup\om^\eps_2\cup\om^\eps_3\quad\mbox{ with }\quad\begin{array}{|l}
\om^\eps_1:=(-1/2;1/2)\times(-\eps/2;\eps/2)^2\\[2pt]
\om^\eps_2:=(-\eps/2;\eps/2)\times(-1/2;1/2)\times(-\eps/2;\eps/2)\\[2pt]
\om^\eps_3:=(-\eps/2;\eps/2)^2\times(-1/2;1/2).
\end{array}
\]
Problem (\ref{PbSpectralCell}) is formally selfadjoint for any value of the dual variable $\eta\in\R^3$. Additionally it is $2\pi$-periodic with respect to each of the $\eta_j$ because the transformation $\eta_j\mapsto\eta_j+2\pi$ leaves invariant the quasiperiodicity conditions. For any $\eta\in[0;2\pi)^3$, the spectrum of (\ref{PbSpectralCell}) is discrete, made of a monotone increasing positive sequence of eigenvalues
\[
0<\Lambda_1^\eps(\eta)\le \Lambda_2^\eps(\eta)\le \dots \le \Lambda_p^\eps(\eta)\le \dots
\]
where the $\Lambda_p^\eps(\eta)$ are counted according to their multiplicity. The functions 
\[
\eta\mapsto \Lambda_p^\eps(\eta)
\]
are continuous (\cite[Chap. 9]{Kato95}) so that the sets 
\begin{equation}\label{SpectralBands}
\Upsilon^\eps_p=\{\Lambda_p^\eps(\eta),\,\eta\in[0;2\pi)^3\}
\end{equation}
are connected compact segments. Finally, according to the theory (see again \cite{Gelf50,Kuch82,Skri87,Naza99a,Kuch93}), the spectrum of the operator $A^{\eps}$ has the form
\[
\sigma(A^\eps)=\bigcup_{p\in\N^\ast}\Upsilon^\eps_p
\]
where $\N^\ast:=\{1,2,\dots\}$. At this stage, we see that to clarify the behaviour of the spectrum of $A^\eps$ with respect to $\eps\to0^+$, we need to study the dependence of the $\Upsilon^\eps_p$ with respect to $\eps$.\\
\newline
As already mentioned in the introduction, the analysis developed for example in \cite{Grie08,Naza05,Naza14} shows that the asymptotic behaviour of the $\Upsilon^\eps_p$ with respect to $\eps$ depends on the features of the Dirichlet Laplacian in the geometry obtained when zooming at the junction region of the periodicity cell $\om^\eps$. More precisely, introduce the unbounded domain
\begin{equation}\label{DefOmega}
\Om:=L_1\cup L_2\cup L_3
\end{equation}
(see Figure \ref{PeriodicLattice} right) where the $L_j$ are the cylinders with unit square cross section appearing in (\ref{defBars}). In $\Om$, consider the Dirichlet spectral problem for the Laplace operator 
\begin{equation}\label{PbSpectralZoom}
\begin{array}{|rcll}
-\Delta v&=&\mu\,v &\mbox{ in }\Om\\[3pt]
v&=&0&\mbox{ on } \partial\Om
\end{array}
\end{equation}
which is now independent of $\eps$. We denote by $A^\Om$ the unbounded, positive definite, selfadjoint operator naturally associated with this problem defined in $\mL^2(\Om)$ and of domain $\mathcal{D}(A^\Om)\subset\mH^1_0(\Om)$. Its continuous spectrum $\sigma_c(A^\Om)$ occupies the ray $[2\pi^2;+\infty)$ and the threshold point $\Lambda_\dagger:=2\pi^2$ is the first eigenvalue of the Dirichlet problem in the cross sections of the branches of $\Om$ (which are unit squares). The main goal of this article is to show the following results.
\begin{theorem}\label{ThmDiscreteSpectrum}
The discrete spectrum of the operator $A^{\Om}$ contains exactly one eigenvalue $\mu_1\in(0;2\pi^2)$.
\end{theorem}
\begin{remark}
Numerically, in Section \ref{SectionNumerics}, we find $\mu_1\approx 12.9\approx1.3\pi^2$. Note that the eigenfunctions associated with $\mu_1$ decay at infinity as $O(-\sqrt{2\pi^2-\mu_1}\,|x_j|)$ in $L_j$, $j=1,2,3$.
\end{remark}
\noindent As classical in literature, we shall say that there is a threshold resonance for Problem (\ref{PbSpectralZoom}) if there is non trivial function which solves (\ref{PbSpectralZoom}) with the threshold value $\mu=2\pi^2$ of the spectral parameter.
\begin{theorem}\label{ThmNoBoundedSol}
There is no threshold resonance for Problem (\ref{PbSpectralZoom}). 
\end{theorem}
\noindent From Theorems \ref{ThmDiscreteSpectrum} and \ref{ThmNoBoundedSol}, we will derive the final result for the spectrum of $A^\eps$ in the initial unbounded periodic lattice:
\begin{theorem}\label{MainThmPerio}
There are positive values $\eps_p>0$ and $c_p>0$ such that for the spectral bands introduced in (\ref{SpectralBands}), we have the estimates
\[
\begin{array}{ll}
\Upsilon^\eps_1\subset \eps^{-2}\mu_1+\eps^{-2}e^{-\sqrt{2\pi^2-\mu_1}/\eps}(-c_1;c_1),&\  \eps\in(0;\eps_1]\\[8pt]
\Upsilon^\eps_{1+q+3(p-1)}\subset \eps^{-2}2\pi^2+p^2\pi^2+\eps\,(-c_p;c_p),&\  \eps\in(0;\eps_p],\  q=1,2,3,\  p\in\N^\ast.\\
\end{array}
\]
\end{theorem}
\noindent Let us comment this result. First, as already mentioned in the introduction, the spectrum of the operator $A^\eps$ goes to $+\infty$ as $\eps\to0$. Additionally this spectrum becomes very sparse. Indeed Theorem \ref{MainThmPerio} implies the following results. The length of the spectral bands are infinitesimal as $\eps\to0$.  Moreover, between the bands $\Upsilon^\eps_1$ and $\Upsilon^\eps_2$, there is a gap, that is a segment of spectral parameters $\Lambda$ such that waves cannot propagate, of size $O(\eps^{-2}(2\pi^2-\mu_1))$ while between $\Upsilon^\eps_{1+q+3(p-1)}$ and $\Upsilon^\eps_{1+q+3p}$, the gap is of width $O(\pi^2(2p-1))$. Therefore the main message here is that in the thin lattice $\Pi^\eps$ the propagation of waves is hampered and occurs for very narrow intervals of frequencies.

\section{Properties of the discrete spectrum of $A^{\Omega}$}\label{Section_DiscreteSpectrum}
The goal of this section is to prove Theorem \ref{ThmDiscreteSpectrum}. We start by showing that the discrete spectrum of $A^{\Omega}$ is non empty (for related multidimensional problems, see \cite{Naza12a}). 
\begin{proposition}\label{PropoAtLeastOne}
The discrete spectrum of $A^{\Omega}$ has at least one eigenvalue.
\end{proposition}
\begin{proof}
Define the 2D X-shaped geometry $\Om^{2D}:=\Om^{2D}_1\cup\Om^{2D}_2$ with $\Om^{2D}_1:=\R\times(-1/2;1/2)$ and $\Om^{2D}_2:=(-1/2;1/2)\times\R$ (see Figure \ref{2Dexamples} right). According to \cite{ScRW89,ABGM91} (see also \cite[Thm. 2.1]{Naza14b} as well as the discussion at the end of this proof), we know that the Dirichlet Laplacian in $\Om^{2D}$ admits exactly one eigenvalue $\mu^{2D}$ below the continuous spectrum which coincides with $[\pi^2;+\infty)$. Let $\varphi\in\mH^1_0(\Om^{2D})$ be a corresponding eigenfunction. In the domain $\Om\subset\R^3$ introduced in (\ref{DefOmega}), consider the function $v$ such that
\[
\begin{array}{ll}
v(x_1,x_2,x_3) = \varphi(x_1,x_2)\cos(\pi x_3)\ \mbox{ in }L_1\cup L_2,\qquad\qquad &v(x_1,x_2,x_3)=0\ \mbox{ in }\Om\setminus\overline{L_1\cup L_2}.
\end{array}
\]
We have 
\[
\int_{\Om}|\nabla v|^2\,dx=\int_{L_1\cup L_2}|\nabla v|^2\,dx=-\int_{L_1\cup L_2}\Delta v v\,dx=(\mu^{2D}+\pi^2)\int_{\Om} v^2\,dx.
\]
But according to the max-min principle (cf. \cite[Thm. 10.2.2]{BiSo87}), we know that
\begin{equation}\label{CaracInfSup}
\inf \sigma(A^{\Om})=\inf_{w\in\mH^1_0(\Om)\setminus\{0\}}\cfrac{\dsp\int_{\Om}|\nabla w|^2\,dx}{\dsp\int_{\Om}w^2\,dx}\,.
\end{equation}
Inserting the above $v$ in the right hand side of (\ref{CaracInfSup}), we deduce that the discrete spectrum of $A^{\Omega}$ contains an eigenvalue below $\mu^{2D}+\pi^2$.\\
\newline
We end this proof by describing the elegant trick proposed in \cite{ScRW89,ABGM91} to show the existence of an eigenvalue in the discrete spectrum of the Dirichlet Laplacian in $\Om^{2D}$. Consider the square coloured in dark gray of Figure \ref{2Dexamples} right. It has side $\sqrt{2}$. Therefore the first eigenvalue of the Dirichlet Laplacian in this geometry is equal to $\pi^2$. Hence, in any domain containing strictly this square and included in $\Om^{2D}$, the first eigenvalue of the Dirichlet Laplacian is strictly less than $\pi^2$. Extending a corresponding eigenfunction by zero to $\Om^{2D}$ and using the max-min principle, we infer that the discrete spectrum of the Dirichlet Laplacian in $\Om^{2D}$ is not empty.
\end{proof}

\noindent It remains to show that the discrete spectrum of $A^{\Om}$ has at most one eigenvalue. To proceed, first we establish a result of symmetry. 
\begin{lemma}\label{LemmaSym}
Let $v\in\mH^1_0(\Om)$ be an eigenfunction of the operator $A^{\Om}$ associated with an eigenvalue $\mu<2\pi^2$. Then $v$ is symmetric with respect to the three planes $(Ox_2x_3)$, $(Ox_3x_1)$, $(Ox_1x_2)$.
\end{lemma}
\begin{proof}
We present the proof of symmetry with respect to the plane $(Ox_2x_3)$. The two other symmetries can be established similarly. Introduce the function $\varphi$ such that
\[
\varphi(x_1,x_2,x_3)=v(x_1,x_2,x_3)-v(-x_1,x_2,x_3).
\]
Our goal is to prove that $\varphi\equiv0$. Set $\Om^+_1:=\{(x_1,x_2,x_3)\in\Om\,|\,x_1>0\}$. The function $\varphi$ satisfies  $-\Delta\varphi=\mu\varphi$ in $\Om^+_1$ and $\varphi=0$ on $\partial\Om^+_1$. Therefore we have 
\begin{equation}\label{IdentityEigen}
\int_{\Om^+_1}|\nabla\varphi|^2\,dx=\mu\int_{\Om^+_1}\varphi^2\,dx.
\end{equation}

\begin{figure}[!ht]
\centering\begin{tikzpicture}[scale=0.75]
\pgfmathsetmacro{\cubexa}{8}
\pgfmathsetmacro{\cubeya}{1}
\pgfmathsetmacro{\cubeza}{1}
\begin{scope}[rotate=90]
\draw[black,dotted] (0,-\cubexa/2,-0.5) -- ++(0,-0.8,0);

\draw[black,fill=gray!30] (0.5,0.5,-1) -- ++(0,0,-\cubexa/2) -- ++(-\cubeya/2,0,0) -- ++(0,0,\cubexa/2) -- cycle;
\draw[black,fill=gray!30] (0.5,0.5,-1) -- ++(0,-\cubeza,0) -- ++(-\cubeya/2,0,0) -- ++(0,\cubeza,0) -- cycle;
\draw[black,fill=gray!30] (0.5,0.5,-1) -- ++(0,0,-\cubexa/2) -- ++(0,-\cubeza,0) -- ++(0,0,\cubexa/2) -- cycle;

\draw[black,fill=gray!30] (0.5,\cubexa/2+1,0.5) -- ++(0,-\cubexa/2,0) -- ++(0,0,-\cubeya) -- ++(0,\cubexa/2,0) -- cycle;
\draw[black,fill=gray!30] (0.5,\cubexa/2+1,0.5) -- ++(-\cubeza/2,0,0) -- ++(0,0,-\cubeya) -- ++(\cubeza/2,0,0) -- cycle;
\draw[black,fill=gray!30] (0.5,\cubexa/2+1,0.5) -- ++(0,-\cubexa/2,0) -- ++(-\cubeza/2,0,0) -- ++(0,\cubexa/2,0) -- cycle;

\draw[black,fill=gray!30] (\cubexa/2+1,0.5,0.5) -- ++(-\cubexa/2,0,0) -- ++(0,-\cubeya,0) -- ++(\cubexa/2,0,0) -- cycle;
\draw[black,fill=gray!30] (\cubexa/2+1,0.5,0.5) -- ++(0,0,-\cubeza) -- ++(0,-\cubeya,0) -- ++(0,0,\cubeza) -- cycle;
\draw[black,fill=gray!30] (\cubexa/2+1,0.5,0.5) -- ++(-\cubexa/2,0,0) -- ++(0,0,-\cubeza) -- ++(\cubexa/2,0,0) -- cycle;

\draw[black,fill=gray!30] (0.5,0.5,0.5) -- ++(-0.5,0,0) -- ++(0,-\cubeya,0) -- ++(0.5,0,0) -- cycle;
\draw[black,fill=gray!30] (0.5,0.5,0.5) -- ++(0,0,-\cubeza) -- ++(0,-\cubeya,0) -- ++(0,0,\cubeza) -- cycle;
\draw[black,fill=gray!30] (0.5,0.5,0.5) -- ++(-0.5,0,0) -- ++(0,0,-\cubeza) -- ++(0.5,0,0) -- cycle;

\draw[black,fill=gray!30] (0.5,-1,0.5) -- ++(0,-\cubexa/2,0) -- ++(0,0,-\cubeya) -- ++(0,\cubexa/2,0) -- cycle;
\draw[black,fill=gray!30] (0.5,-1,0.5) -- ++(-\cubeza/2,0,0) -- ++(0,0,-\cubeya) -- ++(\cubeza/2,0,0) -- cycle;
\draw[black,fill=gray!30] (0.5,-1,0.5) -- ++(0,-\cubexa/2,0) -- ++(-\cubeza/2,0,0) -- ++(0,\cubexa/2,0) -- cycle;

\draw[black,fill=gray!30] (0.5,0.5,\cubexa/2+1) -- ++(0,0,-\cubexa/2) -- ++(-\cubeya/2,0,0) -- ++(0,0,\cubexa/2) -- cycle;
\draw[black,fill=gray!30] (0.5,0.5,\cubexa/2+1) -- ++(0,-\cubeza,0) -- ++(-\cubeya/2,0,0) -- ++(0,\cubeza,0) -- cycle;
\draw[black,fill=gray!30] (0.5,0.5,\cubexa/2+1) -- ++(0,0,-\cubexa/2) -- ++(0,-\cubeza,0) -- ++(0,0,\cubexa/2) -- cycle;

\draw[black,dotted] (\cubexa/2+1,0.5,0.5) -- ++(0.8,0,0);
\draw[black,dotted] (\cubexa/2+1,-0.5,0.5) -- ++(0.8,0,0);
\draw[black,dotted] (\cubexa/2+1,0.5,-0.5) -- ++(0.8,0,0);
\draw[black,dotted] (\cubexa/2+1,-0.5,-0.5) -- ++(0.8,0,0);

\draw[black,dotted] (0.5,\cubexa/2+1,0.5) -- ++(0,0.8,0);
\draw[black,dotted] (0.5,\cubexa/2+1,-0.5) -- ++(0,0.8,0);
\draw[black,dotted] (0,\cubexa/2+1,0.5) -- ++(0,0.8,0);
\draw[black,dotted] (0,\cubexa/2+1,-0.5) -- ++(0,0.8,0);
\draw[black,dotted] (0.5,-\cubexa/2-1,0.5) -- ++(0,-0.8,0);
\draw[black,dotted] (0.5,-\cubexa/2-1,-0.5) -- ++(0,-0.8,0);
\draw[black,dotted] (0,-\cubexa/2-1,0.5) -- ++(0,-0.8,0);

\draw[black,dotted] (0.5,0.5,\cubexa/2+1) -- ++(0,0,0.8);
\draw[black,dotted] (0,0.5,\cubexa/2+1) -- ++(0,0,0.8);
\draw[black,dotted] (0.5,-0.5,\cubexa/2+1) -- ++(0,0,0.8);
\draw[black,dotted] (0,-0.5,\cubexa/2+1) -- ++(0,0,0.8);
\draw[black,dotted] (0.5,0.5,-\cubexa/2-1) -- ++(0,0,-0.8);
\draw[black,dotted] (0,0.5,-\cubexa/2-1) -- ++(0,0,-0.8);
\draw[black,dotted] (0.5,-0.5,-\cubexa/2-1) -- ++(0,0,-0.8);

\node at (3,0,0.5){$L_1^+$};
\node at (0,2.6,-1){$S_2^+$};
\node at (0,-3.6,-1){$S_2^-$};
\node at (0,-0.5,2){$S_3^+$};
\node at (0,-0.5,-4.5){$S_3^-$};
\draw[black,->] (1.6,-1.6,0) -- ++(-1,1,0);
\node at (1.8,-2,0){$Q^+_1$};
\end{scope}
\end{tikzpicture}\qquad \begin{tikzpicture}[scale=0.75]
\pgfmathsetmacro{\cubexa}{8}
\pgfmathsetmacro{\cubeya}{1}
\pgfmathsetmacro{\cubeza}{1}
\begin{scope}[rotate=90]
\draw[black,fill=gray!30] (0.5,\cubexa/2+1,0.5) -- ++(0,-\cubexa/2,0) -- ++(0,0,-\cubeya/2) -- ++(0,\cubexa/2,0) -- cycle;
\draw[black,fill=gray!30] (0.5,\cubexa/2+1,0.5) -- ++(-\cubeza/2,0,0) -- ++(0,0,-\cubeya/2) -- ++(\cubeza/2,0,0) -- cycle;
\draw[black,fill=gray!30] (0.5,\cubexa/2+1,0.5) -- ++(0,-\cubexa/2,0) -- ++(-\cubeza/2,0,0) -- ++(0,\cubexa/2,0) -- cycle;

\draw[black,fill=gray!30] (\cubexa/2+1,0.5,0.5) -- ++(-\cubexa/2,0,0) -- ++(0,-\cubeya/2,0) -- ++(\cubexa/2,0,0) -- cycle;
\draw[black,fill=gray!30] (\cubexa/2+1,0.5,0.5) -- ++(0,0,-\cubeza/2) -- ++(0,-\cubeya/2,0) -- ++(0,0,\cubeza/2) -- cycle;
\draw[black,fill=gray!30] (\cubexa/2+1,0.5,0.5) -- ++(-\cubexa/2,0,0) -- ++(0,0,-\cubeza/2) -- ++(\cubexa/2,0,0) -- cycle;

\draw[black,fill=gray!30] (0.5,0.5,0.5) -- ++(-0.5,0,0) -- ++(0,-\cubeya/2,0) -- ++(0.5,0,0) -- cycle;
\draw[black,fill=gray!30] (0.5,0.5,0.5) -- ++(0,0,-\cubeza/2) -- ++(0,-\cubeya/2,0) -- ++(0,0,\cubeza/2) -- cycle;
\draw[black,fill=gray!30] (0.5,0.5,0.5) -- ++(-0.5,0,0) -- ++(0,0,-\cubeza/2) -- ++(0.5,0,0) -- cycle;

\draw[black,fill=gray!30] (0.5,0.5,\cubexa/2+1) -- ++(0,0,-\cubexa/2) -- ++(-\cubeya/2,0,0) -- ++(0,0,\cubexa/2) -- cycle;
\draw[black,fill=gray!30] (0.5,0.5,\cubexa/2+1) -- ++(0,-\cubeza/2,0) -- ++(-\cubeya/2,0,0) -- ++(0,\cubeza/2,0) -- cycle;
\draw[black,fill=gray!30] (0.5,0.5,\cubexa/2+1) -- ++(0,0,-\cubexa/2) -- ++(0,-\cubeza/2,0) -- ++(0,0,\cubexa/2) -- cycle;

\draw[black,dotted] (\cubexa/2+1,0.5,0.5) -- ++(0.8,0,0);
\draw[black,dotted] (\cubexa/2+1,0,0.5) -- ++(0.8,0,0);
\draw[black,dotted] (\cubexa/2+1,0.5,0) -- ++(0.8,0,0);
\draw[black,dotted] (\cubexa/2+1,0,0) -- ++(0.8,0,0);

\draw[black,dotted] (0.5,\cubexa/2+1,0.5) -- ++(0,0.8,0);
\draw[black,dotted] (0.5,\cubexa/2+1,0) -- ++(0,0.8,0);
\draw[black,dotted] (0,\cubexa/2+1,0.5) -- ++(0,0.8,0);
\draw[black,dotted] (0,\cubexa/2+1,0) -- ++(0,0.8,0);

\draw[black,dotted] (0.5,0.5,\cubexa/2+1) -- ++(0,0,0.8);
\draw[black,dotted] (0,0.5,\cubexa/2+1) -- ++(0,0,0.8);
\draw[black,dotted] (0.5,0,\cubexa/2+1) -- ++(0,0,0.8);
\draw[black,dotted] (0,0,\cubexa/2+1) -- ++(0,0,0.8);
\draw[black,->] (1.6,1.6,0) -- ++(-1,-1,0);
\node at (1.8,1.8,0){$Q^+$};
\end{scope}
\end{tikzpicture}
\caption{Left: exploded decomposition of $\Om^+_1$. Right: exploded decomposition of $\Om^+$. \label{ExplodedDomain}}
\end{figure}
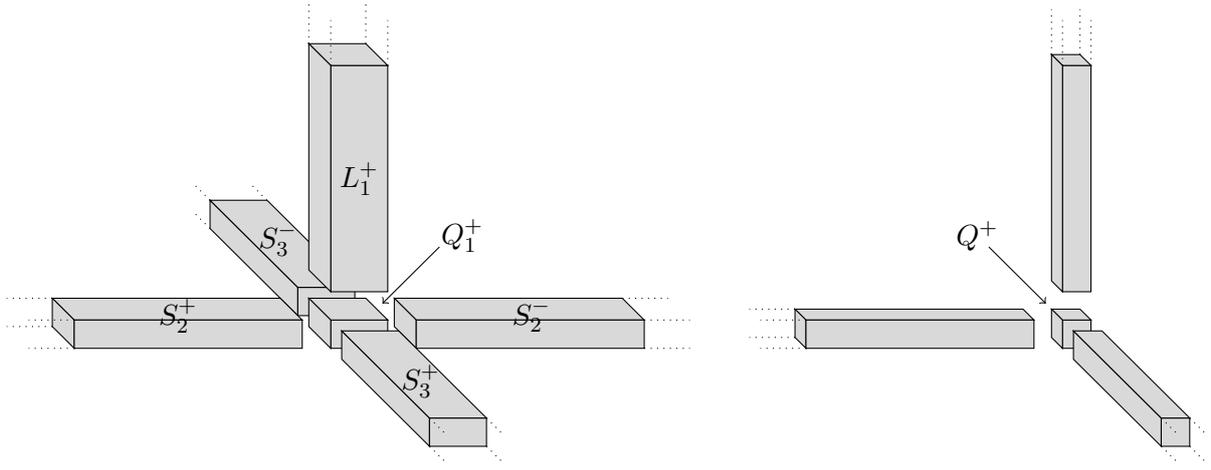

\noindent Define the domains 
\[
\begin{array}{rcl}
L^+_1 & := & \{x\in L_1\,|\,x_1>1/2\}\\[2pt]
S_2 & := & \{x\in L_2\,|\,x_1>0\}=\Om^+_1\cap L_2\\[2pt]
S_3 & := & \{x\in L_3\,|\,x_1>0\}=\Om^+_1\cap L_3\\[2pt]
S^{\pm}_2 & := & \{x\in S_2\,|\,\pm x_2>1/2\}\\[2pt]
S^{\pm}_3 & := & \{x\in S_3\,|\,\pm x_3>1/2\}\\[2pt]
Q^+_1& := & \{x\in\R^3\,|\,0<x_1<1/2,\,|x_2|<1/2,\,|x_3|<1/2\}=S_2\cap S_3
\end{array}
\]
(see the exploded view of Figure \ref{ExplodedDomain} left). Using the Poincar\'e inequality in a section of $L^+_1$, which is a square of unit side, and then integrating with respect to the $x_1$ variable, first we find
\begin{equation}\label{PoincareTranche0}
2\pi^2 \int_{L^+_1}\varphi^2\,dx \le \int_{L^+_1}|\nabla\varphi|^2\,dx.
\end{equation}
In $S^\pm_2$, $S^\pm_3$, since the transverse section is smaller than in $L^+_1$, we have better estimates, namely
\begin{equation}\label{PoincareTranche}
5\pi^2\int_{S_2^\pm}\varphi^2\,dx \le \int_{S_2^\pm}(\partial_{x_1}\varphi)^2+(\partial_{x_3}\varphi)^2\,dx,\qquad 5\pi^2\int_{S_3^\pm}\varphi^2\,dx \le \int_{S_3^\pm}(\partial_{x_1}\varphi)^2+(\partial_{x_2}\varphi)^2\,dx.
\end{equation}
The difficulty to complete the proof however is that in $Q^+_1$, using a 1D  Poincar\'e inequality with respect to the $x_1$ variable, we can simply get 
\begin{equation}\label{PoincareTranche0bis}
\pi^2 \int_{Q^+_1}\varphi^2\,dx \le \int_{Q^+_1}(\partial_{x_1}\varphi)^2\,dx.
\end{equation}
This is enough to conclude directly that $\varphi\equiv0$ when $\mu\in(0;\pi^2)$ but not in the general situation $\mu\in(0;2\pi^2)$ that we wish to deal with. Therefore we have to refine the strategy. What we will do is to exploit the ``extra'' terms coming from the efficiency of estimates (\ref{PoincareTranche}) to control the $\mL^2$-norm of $\varphi$ in $Q^+_1$. More precisely, the Lemma 5.1 of \cite{BaMaNa17} that we recall in Appendix (Lemma \ref{LemmaPoincareFriedrichs}) guarantees that for a given $a>0$, we have the Friedrichs inequality 
\begin{equation}\label{PoincareModified}
\kappa(a)\int_{0}^{1/2}\phi^2\,dt \le \int_{0}^{+\infty}(\partial_t\phi)^2\,dt+a^2\int_{1/2}^{+\infty}\phi^2\,dt,\qquad \forall \phi\in\mH^1(0;+\infty),
\end{equation}
where $\kappa(a)$ is the smallest positive root of the transcendental equation
\begin{equation}\label{TransEqn}
\sqrt{\kappa}\tan\bigg(\cfrac{\sqrt{\kappa}}{2}\bigg)=a.
\end{equation}
In particular, solving (\ref{TransEqn}) with $a=\pi\sqrt{5/2}$, we find $\kappa(a)>\pi^2/2$. Therefore, using (\ref{PoincareModified}), we can write
\begin{equation}\label{PoincareTranche01}
\cfrac{\pi^2}{2}\int_{Q^+_1}\varphi^2\,dx \le \int_{S_2}(\partial_{x_2}\varphi)^2\,dx+\cfrac{5\pi^2}{2}\int_{S_2^+\cup S_2^-}\varphi^2\,dx
\end{equation}
as well as 
\begin{equation}\label{PoincareTranche02}
\cfrac{\pi^2}{2}\int_{Q^+_1}\varphi^2\,dx \le \int_{S_3}(\partial_{x_3}\varphi)^2\,dx+\cfrac{5\pi^2}{2}\int_{S_3^+\cup S_3^-}\varphi^2\,dx.
\end{equation}
Then inserting (\ref{PoincareTranche}) in (\ref{PoincareTranche01}), (\ref{PoincareTranche02}) and summing up the resulting estimates, we obtain
\begin{equation}\label{PoincareTranche1}
\begin{array}{rcl}
\pi^2\dsp\int_{Q^+_1}\varphi^2\,dx &\le&\dsp \int_{S_2}(\partial_{x_2}\varphi)^2\,dx+\cfrac{1}{2}\int_{S_2^+\cup S_2^-}(\partial_{x_3}\varphi)^2\,dx+\dsp \int_{S_3}(\partial_{x_3}\varphi)^2\,dx+\cfrac{1}{2}\int_{S_3^+\cup S_3^-}(\partial_{x_2}\varphi)^2\,dx\\[10pt]
 & & +\dsp\cfrac{1}{2}\int_{S_2^+\cup S_2^-\cup S_3^+\cup S_3^-}(\partial_{x_1}\varphi)^2\,dx.
\end{array}
\end{equation}
On the other hand, (\ref{PoincareTranche}) also yields
\begin{equation}\label{PoincareTranche2}
2\pi^2\int_{S^+_2\cup S^-_2\cup S^+_3\cup S^-_3}\varphi^2\,dx \le \cfrac{1}{2} \int_{S^+_2\cup S^-_2\cup S^+_3\cup S^-_3}(\partial_{x_1}\varphi)^2\,dx+\cfrac{1}{2}\int_{S_2^+\cup S_2^-}(\partial_{x_3}\varphi)^2\,dx+\cfrac{1}{2}\int_{S_3^+\cup S_3^-}(\partial_{x_2}\varphi)^2\,dx.
\end{equation}
Finally, summing up (\ref{PoincareTranche0}), (\ref{PoincareTranche0bis}), (\ref{PoincareTranche1}) and (\ref{PoincareTranche2}), we get 
\[
2\pi^2 \int_{\Om^+_1}\varphi^2\,dx \le \int_{\Om^+_1}|\nabla\varphi|^2\,dx.
\]
This estimate together with the identity (\ref{IdentityEigen}) imply $\varphi\equiv0$.
\end{proof}
\noindent We can now establish the main result of this section.\\[10pt]
\textit{Proof of Theorem \ref{ThmDiscreteSpectrum}.} From Proposition \ref{PropoAtLeastOne}, we know that there is at least one eigenvalue $\mu_1$ of $A^{\Om}$ below $2\pi^2$. Assume that $A^{\Om}$ has a second eigenvalue $\mu_2$ such that $\mu_2<2\pi^2$. Set $\Om^+:=\{x\in\Om\,|\,x_1>0,\,x_2>0,\,x_3>0\}$ (see Figure \ref{ExplodedDomain} right). Then according to the result of symmetry of Lemma \ref{LemmaSym}, the problem 
\[
\begin{array}{|rcll}
-\Delta v &=&\mu v  &\mbox{in }\Om^+\\[3pt]
v&=&0 &\mbox{on }\Sigma_0:=\partial\Om\cap \partial \Om^+\\[3pt]
\partial_nv&=&0 &\mbox{on }\partial \Om^+\setminus\Sigma_0
\end{array}
\]
admits the two eigenvalues $\mu_1$, $\mu_2$. Besides, from the max-min principle (\cite[Thm. 10.2.2]{BiSo87}), we have
\[
\mu_2=\max_{E\subset\mathscr{E}_1}\inf_{w\in E\setminus\{0\}}\cfrac{\dsp\int_{\Om^+}|\nabla w|^2\,dx}{\dsp\int_{\Om^+} w^2\,dx}\,
\]
where $\mathscr{E}_1$ denotes the set of subspaces of $\mH^1_{0}(\Om^+;\Sigma_0):=\{\varphi\in\mH^1(\Om^+)\,|\,\varphi=0\mbox{ on }\Sigma_0\}$ of codimension one. In particular, we have 
\begin{equation}\label{Contra1}
\mu_2 \ge \inf_{w\in E\setminus\{0\}}\cfrac{\dsp\int_{\Om^+}|\nabla w|^2\,dx}{\dsp\int_{\Om^+} w^2\,dx}
\end{equation}
with $E=\{\varphi\in\mH^1_{0}(\Om^+;\Sigma_0)\,|\,\int_{Q^+}\varphi\,dx=0\}$, $Q^+:=(0;1/2)^3$. However from the Poincar\'e inequality, we can write for $w\in E\setminus\{0\}$, 
\begin{equation}\label{IdentityEigenD}
2\pi^2\int_{\Om^+\setminus\overline{Q^+}}w^2\,dx \le \int_{\Om^+\setminus\overline{Q^+}} |\nabla w|^2\,dx
\end{equation}
and there holds, according to the max-min principle, 
\begin{equation}\label{Contra2}
\inf_{w\in E\setminus\{0\}}\cfrac{\dsp\int_{Q^+}|\nabla w|^2\,dx}{\dsp\int_{Q^+} w^2\,dx}=4\pi^2
\end{equation}
because the first positive eigenvalue of the Neumann Laplacian in $Q^+$ is equal to $4\pi^2$. Using (\ref{IdentityEigenD}) and (\ref{Contra2}) in (\ref{Contra1}) leads to $\mu_2\ge2\pi$ which contradicts the initial assumption. Therefore $A^{\Om}$ cannot have two eigenvalues below the continuous spectrum.\hfill $\square$

\section{Absence of threshold resonance}\label{sectionAbsenceofBoundedSol}

In this section, we establish Theorem \ref{ThmNoBoundedSol}. To proceed, we apply the tools of \cite{Pank17,BaNa21} that we recall now. For $R\ge1/2$ and $j=1,2,3$, define the truncated cylinder
\[
L^R_j  :=  \{x\in L_j\,|\,|x_j|<R\}
\]
and set $\Om^R:=L^R_1\cup L^R_2 \cup L^R_3$. In $\Om^R$, consider the spectral problem with mixed boundary conditions
\begin{equation}\label{SpectralRMixed}
\begin{array}{|rcll}
-\Delta v & = & \mu v &\mbox { in }\Om^R\\[2pt]
 v & = & 0 &\mbox{ on }\partial\Om^R\cap\partial\Om\\[2pt]
\partial_n v & = & 0 &\mbox{ on }\partial\Om^R\setminus\partial\Om
\end{array}
\end{equation}
where $n$ stand for the unit normal vector to $\partial\Om^R$ directed to the exterior of $\Om^R$. Set $\mH^1_{0}(\Om^R;\partial\Om^R\cap\partial\Om):=\{\varphi\in\mH^1(\Om^R)\,|\,\varphi=0\mbox{ on }\partial\Om^R\cap\partial\Om\}$ and finally denote by $B^R$ the unbounded operator associated with (\ref{SpectralRMixed}) of domain $\mathcal{D}(B^R)\subset\mH^1_{0}(\Om^R;\partial\Om^R\cap\partial\Om)$.\\
\newline
The Theorem 3 of \cite{Pank17} (see also the necessary and sufficient condition of \cite[Thm. 1]{BaNa21}) guarantees that there is no threshold resonant for Problem (\ref{PbSpectralZoom}) if there exists for some $R\ge0$ such that the second eigenvalue $\mu^R_2$ of $B_R$ satisfies $\mu^R_2>2\pi^2$.\\
\newline
Therefore from now our goal is to show that the second eigenvalue of $B_R$ is larger than $2\pi^2$ (Proposition \ref{PropoBoundedSecond} below). We start with a result of symmetry similar to Lemma \ref{LemmaSym}.

\begin{lemma}\label{LemmaSymR}
For $R$ large enough, if $v\in\mH^1_{0}(\Om^R;\partial\Om^R\cap\partial\Om)$ is an eigenfunction of the operator $B^R$ associated with an eigenvalue $\mu\le2\pi^2$, then $v$ is symmetric with respect to the planes $(Ox_2x_3)$, $(Ox_3x_1)$, $(Ox_1x_2)$.
\end{lemma}
\begin{proof}
The demonstration follows the lines of the one of Lemma \ref{LemmaSym}. However we write the details for the sake of clarity. We focus our attention on the symmetry with respect to the plane $(Ox_2x_3)$, the two other ones being similar. Introduce the function $\varphi$ such that
\[
\varphi(x_1,x_2,x_3)=v(x_1,x_2,x_3)-v(-x_1,x_2,x_3).
\]
We wish to show that $\varphi\equiv0$. Set $\Om^{R+}_1:=\{x\in\Om^R\,|\,x_1>0\}$. We have 
\begin{equation}\label{IdentityEigenR}
\int_{\Om^{R+}_1}|\nabla\varphi|^2\,dx=\mu\int_{\Om^{R+}_1}\varphi^2\,dx.
\end{equation}
Define the domains 
\[
\begin{array}{rcl}
L^{R+}_1 & := & \{x\in L^R_1\,|\,x_1>1/2\}\\[2pt]
S^R_2 & := & \{x\in L^R_2\,|\,x_1>0\}=\Om^{R+}_1\cap L^R_2\\[2pt]
S^R_3 & := & \{x\in L^R_3\,|\,x_1>0\}=\Om^{R+}_1\cap L^R_3\\[2pt]
S^{R\pm}_2 & := & \{x\in S^R_2\,|\,\pm x_2>1/2\}\\[2pt]
S^{R\pm}_3 & := & \{x\in S^R_3\,|\,\pm x_3>1/2\}.
\end{array}
\]
In $L^{R+}_1$ the Poincar\'e inequality gives
\begin{equation}\label{PoincareTranche0R}
2\pi^2 \int_{L^{R+}_1}\varphi^2\,dx \le \int_{L^{R+}_1}|\nabla\varphi|^2\,dx.
\end{equation}
On the other hand, in $Q^+_1=(0;1/2)\times(-1/2;1/2)^2$ there holds
\begin{equation}\label{PoincareTranche0bisR}
\pi^2 \int_{Q^+_1}\varphi^2\,dx \le \int_{Q^+_1}(\partial_{x_1}\varphi)^2\,dx.
\end{equation}
The Lemma \ref{LemmaPoincareFriedrichsR} in Appendix (see also Lemma 5.1 of \cite{BaMaNa17}) guarantees that for $R>1/2$ and $a>0$, we have
\begin{equation}\label{PoincareModifiedR}
 \kappa(a,R)\int_{0}^{1/2}\phi^2\,dt \le \int_{0}^{R}(\partial_t\phi)^2\,dt+a^2\int_{1/2}^{R}\phi^2\,dt,\qquad \forall\phi\in\mH^1(0;R),
\end{equation}
where $\kappa(a,R)>0$ converges to the constant $\kappa(a)$ appearing in (\ref{TransEqn}) as $R\to+\infty$. For $a=\pi\sqrt{5/2}$, as already said, one finds $\kappa(a)>\pi^2/2$. Introduce $\delta>0$ such that $\kappa(\pi\sqrt{5/2})>\pi^2/2+\delta$. We know that there is $R_0$ large enough such that we have $\kappa(\pi\sqrt{5/2},R)>\pi^2/2+\delta/2$ for all $R\ge R_0$. We infer that we have
\begin{equation}\label{PoincareTranche01R}
\cfrac{\pi^2+\delta}{2}\int_{Q^+_1}\varphi^2\,dx \le \int_{S^R_2}(\partial_{x_2}\varphi)^2\,dx+\cfrac{5\pi^2}{2}\int_{S_2^{R+}\cup S_2^{R-}}\varphi^2\,dx
\end{equation}
and
\begin{equation}\label{PoincareTranche02R}
\cfrac{\pi^2+\delta}{2}\int_{Q^+_1}\varphi^2\,dx \le \int_{S^R_3}(\partial_{x_3}\varphi)^2\,dx+\cfrac{5\pi^2}{2}\int_{S_3^{R+}\cup S_3^{R-}}\varphi^2\,dx.
\end{equation}
But from the Poincar\'e inequality in the transverse section of $S_2^{R\pm}$, $S_3^{R\pm}$, we know that 
\begin{equation}\label{PoincareTrancheR}
5\pi^2\int_{S_2^{R\pm}}\varphi^2\,dx \le \int_{S_2^{R\pm}}(\partial_{x_1}\varphi)^2+(\partial_{x_3}\varphi)^2\,dx,\qquad 5\pi^2\int_{S_3^{R\pm}}\varphi^2\,dx \le \int_{S_3^{R\pm}}(\partial_{x_1}\varphi)^2+(\partial_{x_2}\varphi)^2\,dx.
\end{equation}
Inserting (\ref{PoincareTrancheR}) in (\ref{PoincareTranche01R}), (\ref{PoincareTranche02R}) and summing up the resulting estimates, we obtain
\begin{equation}\label{PoincareTranche1R}
\begin{array}{rcl}
(\pi^2+\delta)\dsp\int_{Q^+_1}\varphi^2\,dx &\le&\dsp \int_{S^R_2}(\partial_{x_2}\varphi)^2\,dx+\dsp \int_{S_3^R}(\partial_{x_3}\varphi)^2\,dx\\[10pt]
 & &+\dsp\cfrac{1}{2}\int_{S_2^{R+}\cup S_2^{R-}}(\partial_{x_3}\varphi)^2\,dx+\cfrac{1}{2}\int_{S_3^{R+}\cup S_3^{R-}}(\partial_{x_2}\varphi)^2\,dx\\[10pt]
 & & +\dsp\cfrac{1}{2}\int_{S_2^{R+}\cup S_2^{R-}\cup S_3^{R+}\cup S_3^{R-}}(\partial_{x_1}\varphi)^2\,dx.
\end{array}
\end{equation}
On the other hand, (\ref{PoincareTrancheR}) also yields
\begin{equation}\label{PoincareTranche2R}
\begin{array}{l}
\dsp2\pi^2\int_{S^{R+}_2\cup S^{R-}_2\cup S^{R+}_3\cup S^{R-}_3}\varphi^2\,dx \\[10pt]
\dsp\le \cfrac{1}{2} \int_{S^{R+}_2\cup S^{R-}_2\cup S^{R+}_3\cup S^{R-}_3}(\partial_{x_1}\varphi)^2\,dx+\cfrac{1}{2}\int_{S_2^{R+}\cup S_2^{R-}}(\partial_{x_3}\varphi)^2\,dx+\cfrac{1}{2}\int_{S_3^{R+}\cup S_3^{R-}}(\partial_{x_2}\varphi)^2\,dx.
\end{array}
\end{equation}
Finally, summing up (\ref{PoincareTranche0R}), (\ref{PoincareTranche0bisR}), (\ref{PoincareTranche1R}) and (\ref{PoincareTranche2R}), we get 
\[
\delta\dsp\int_{Q^+_1}\varphi^2\,dx+2\pi^2 \int_{\Om^{R+}_1}\varphi^2\,dx \le \int_{\Om^{R+}_1}|\nabla\varphi|^2\,dx.
\]
This estimate together with the identity (\ref{IdentityEigenR}) imply $\int_{Q^+_1}\varphi^2\,dx=0$ and so $\varphi\equiv0$ in $Q^+_1$. From the unique continuation principle, this gives $\varphi\equiv0$ in $\Om^{R+}_1$.
\end{proof}

\begin{proposition}\label{PropoBoundedSecond}
For $R$ large enough, the second eigenvalue $\mu_2^R$ of $B^R$ satisfies $\mu_2^R>2\pi^2$. 
\end{proposition}
\begin{proof}
One applies Lemma \ref{LemmaSymR} to reduce the analysis to $\Om^{R+}:=\{x\in\Om^R\,|\,x_1>0,\,x_2>0,\,x_3>0\}$. Here in particular we use the assumption that $R$ is large enough. The rest of the proof is completely similar to the one of Theorem \ref{ThmDiscreteSpectrum}. 
\end{proof}

\section{Model problems for the spectral bands}\label{SectionModels}

In this section, we establish Theorem \ref{MainThmPerio} and obtain models for the spectral bands $\Upsilon^\eps_p$ appearing in (\ref{SpectralBands}). We recall that the spectrum of $A^\eps$ is such that $\sigma(A^\eps)=\bigcup_{p\in\N^\ast}\Upsilon^\eps_p$.

\subsection{Study of $\Upsilon^\eps_1$}
By definition, we have 
\begin{equation}\label{Def_FirstBande}
\Upsilon^\eps_1=\{\Lambda_1^\eps(\eta),\,\eta\in[0;2\pi)^3\}
\end{equation}
where $\Lambda_1^\eps(\eta)$ is the first eigenvalue of (\ref{PbSpectralCell}). Therefore our goal is to obtain an asymptotic expansion of $\Lambda_1^\eps(\eta)$ with respect to $\eps$ as $\eps\to0^+$.\\
\newline
Pick $\eta\in[0;2\pi)^3$. Let $u^\eps(\cdot,\eta)$ be an eigenfunction associated with $\Lambda_1^\eps(\eta)$. As a first approximation when $\eps\to0^+$, it is natural to consider the expansion
\begin{equation}\label{ExpansionTR0}
\Lambda_1^\eps(\eta)=\eps^{-2}\mu_1+\dots,\qquad u^\eps(x,\eta)=v_1(x_1/\eps)+\dots
\end{equation}
where $\mu_1\in(0;2\pi^2)$ stands for the eigenvalue of the discrete spectrum of the operator $A^{\Om}$ introduced in Theorem \ref{ThmDiscreteSpectrum} and $v_1\in\mH^1_0(\Om)$ is a corresponding eigenfunction that we choose such that $\|v_1\|_{\mL^2(\Om)}=1$. Indeed, inserting $(\eps^{-2}\mu_1,v_1(\cdot/\eps))$ in Problem (\ref{PbSpectralCell}) only leaves a small discrepancy on the square faces of $\partial\om^\eps$ because $v_1$ is exponentially decaying at infinity. Let us write more precisely the decomposition of $v_1$ at infinity because this will be useful in the sequel. To proceed and keep short notation, we shall work with the coordinates $(z_j,y_j)$, $j=1,\dots,6$, such that
\begin{equation}\label{Def_coor_local}
\begin{array}{|l}
z_1=x_1,\quad y_1=(x_2,x_3);\\[3pt]
z_2=x_2,\quad y_2=(x_3,x_1);\\[3pt]
z_3=x_3,\quad y_3=(x_1,x_2);
\end{array}\qquad\qquad \begin{array}{|l}
z_4=-x_1,\quad y_4=(x_2,-x_3);\\[3pt]
z_5=-x_2,\quad y_5=(x_3,-x_1);\\[3pt]
z_6=-x_3,\quad y_6=(x_1,-x_2).
\end{array}
\end{equation}
We also define for $j=1,\dots,6$, the branch
\[
\mathcal{L}_j:=\{x\in\Om\,|\,z_j>1/2\}.
\]
Then Fourier decomposition together with the result of symmetry of Lemma \ref{LemmaSym} and the fact that $\mu_1$ is a simple eigenvalue\footnote{This is needed to show that $K$ is the same in the branches $\mathcal{L}_j$, $j=1,\dots,6$.} guarantee that for $j=1,\dots,6$, we have 
\[
v_1(x)=K\,e^{-\beta_1 z_j}\,U_{\dagger}(y_j)+O(e^{-\beta_2 z_j})\quad\mbox{ in }\mathcal{L}_j\quad\mbox{ as } z_j\to+\infty.
\]
Here $K\in\R^*$ is independent of $j$, $\beta_1:=\sqrt{2\pi^2-\mu_1}$, $\beta_2:=\sqrt{5\pi^2-\mu_1}$ and 
\begin{equation}\label{defUdagger}
U_{\dagger}(s_1,s_2)=2\cos(\pi s_1)\cos(\pi s_2).
\end{equation}
The first model (\ref{ExpansionTR0}) is interesting but does not comprise the dependence with respect $\eta$. To improve it, consider the more refined ans\"atze
\begin{equation}\label{ExpansionTR1}
\Lambda_1^\eps(\eta)=\eps^{-2}\mu_1+\eps^{-2}e^{-\beta_1/\eps}M(\eta)+\dots,\qquad u^\eps(x,\eta)=v_1(x_1/\eps)+e^{-\beta_1/\eps}V(x_1/\eps,\eta)+\dots
\end{equation}
where the quantities $M(\eta)$, $V(\cdot,\eta)$ are to determine. Inserting (\ref{ExpansionTR1}) into (\ref{PbSpectralCell}), first we obtain that $V(\cdot,\eta)$ must satisfy
\begin{equation}\label{Pb_corrector_TR}
\begin{array}{|rccl}
-\Delta V(\cdot,\eta)-\mu_1V(\cdot,\eta)&=&M(\eta)v_1 & \mbox{ in }\Om\\[3pt]
 V(\cdot,\eta) &=&0 &\mbox{ on }\partial\Om.
\end{array}
\end{equation}
In order to have a non zero solution to (\ref{Pb_corrector_TR}), we must look for a $V(\cdot,\eta)$ which is growing at infinity. Due to (\ref{Pb_corrector_TR}), the simplest growth that we can allow is
\[
V(x,\eta)=B_j\,e^{\beta_1 z_j}\,U_{\dagger}(y_j)+O(e^{-\beta_1 z_j})\quad\mbox{ in }\mathcal{L}_j\quad\mbox{ as } z_j\to+\infty
\]
where the $B_j$ are some constants. Then the quasi-periodic conditions satisfied by $u^\eps(\cdot,\eta)$ at the faces located at $x_1=\pm1/2$, $x_2=\pm1/2$, $x_3=\pm1/2$ (see (\ref{PbSpectralCell})) lead us to choose the $B_j$ such that
\[
\begin{array}{|l}
K+B_4=e^{i\eta_1}(K+B_1)\\[2pt]
K-B_4=e^{i\eta_1}(-K+B_1)
\end{array}\qquad \begin{array}{|l}
K+B_5=e^{i\eta_2}(K+B_2)\\[2pt]
K-B_5=e^{i\eta_2}(-K+B_2)
\end{array}\qquad\begin{array}{|l}
K+B_6=e^{i\eta_3}(K+B_3)\\[2pt]
K-B_6=e^{i\eta_3}(-K+B_3).
\end{array}
\]
Solving these systems, we obtain $B_j=K\,e^{-i\eta_j}$, $B_{j+3}=K\,e^{+i\eta_j}$ for $j=1,2,3$. Now since $\mu_1$ is a simple eigenvalue of $A^\Om$, multiplying (\ref{Pb_corrector_TR}) by $v_1$, integrating by part in $\Om^R$ and taking the limit $R\to+\infty$, we find that there is a solution if and only if the following compatibility condition 
\[
M(\eta)\|v_1\|_{\mL^2(\Om)}=-2\beta_1K^2\sum_{j=1}^3e^{i\eta_j}+e^{-i\eta_j}\qquad\Leftrightarrow\qquad M(\eta)=-4\beta_1K^2\sum_{j=1}^3\cos(\eta_j)
\]
is satisfied. This defines the value of $M(\eta)$ in the expansion (\ref{ExpansionTR1}). From (\ref{Def_FirstBande}), this gives the inclusion $\Upsilon^\eps_1\subset \eps^{-2}\mu_1+\eps^{-2}e^{-\sqrt{2\pi^2-\mu_1}/\eps}(-c_1;c_1)$ for $\eps$ small enough of Theorem \ref{MainThmPerio}. For $c_1$, we can take any constant larger than $12\beta_1K^2$. Note that we decided to focus on a rather formal presentation above for the sake of conciseness. We emphasize that all these results can be completely justified by proving rigorous error. This has been realized in detail in \cite{Naza17,Naza18} for similar problems and can be repeated with obvious modifications.

\subsection{Study $\Upsilon^\eps_k$, $k\ge 2$}

We turn our attention to the asymptotic of the spectral bands of higher frequency
\begin{equation}\label{Def_FirstBande_2}
\Upsilon^\eps_k=\{\Lambda_k^\eps(\eta),\,\eta\in[0;2\pi)^3\},\qquad k\ge2,
\end{equation}
as $\eps\to0^+$. Pick $\eta\in[0;2\pi)^3$ and introduce $u^\eps(\cdot,\eta)$ an eigenfunction associated with $\Lambda_k^\eps(\eta)$. In the sequel, to simplify, we remove the subscript ${}_k$ and do not indicate the dependence on $\eta$. As a first approximation when $\eps\to0^+$, we consider the expansion
\begin{equation}\label{ExpansionTR0_2}
\Lambda^\eps=\eps^{-2}2\pi^2+\nu+\dots,\qquad u^\eps(x)=v^\eps(x)+\dots
\end{equation}
with $v^\eps$ of the form
\[
v^\eps(x)=\begin{array}{|rl}
\gamma^\pm_{1}(x_1)\,U_{\dagger}(x_2/\eps,x_3/\eps)&\mbox{ in }\om^{\eps\pm}_1:=\{x\in\om^{\eps}_1\,|\,\pm x_1>\eps/2\},\\[3pt]
\gamma^\pm_{2}(x_2)\,U_{\dagger}(x_3/\eps,x_1/\eps)&\mbox{ in }\om^{\eps\pm}_2:=\{x\in\om^{\eps}_2\,|\,\pm x_2>\eps/2\},\\[3pt]
\gamma^\pm_{3}(x_3)\,U_{\dagger}(x_1/\eps,x_2/\eps)&\mbox{ in }\om^{\eps\pm}_3:=\{x\in\om^{\eps}_3\,|\,\pm x_3>\eps/2\},\\[3pt]
\end{array}
\]
where the functions $\gamma^\pm_{j}$, $j=1,2,3$, are to determine ($U_{\dagger}$ is defined in (\ref{defUdagger})). Inserting (\ref{ExpansionTR0_2}) into Problem (\ref{PbSpectralCell}), we obtain for $j=1,2,3$,
\begin{equation}\label{farfield}
\begin{array}{|rcl}
\partial^2_s\gamma^+_{j}+\nu\gamma^+_{j}&=&0\quad\mbox{ in }(0;1/2)\\[3pt]
\partial^2_s\gamma^-_{j}+\nu\gamma^-_{j}&=&0\quad\mbox{ in }(-1/2;0)\\[3pt]
\gamma^-_{j}(-1/2)&=&e^{i\eta_j}\gamma^+_{j}(+1/2)\\[3pt]
\partial_s\gamma^-_{j}(-1/2)&=&e^{i\eta_j}\partial_s\gamma^+_{j}(+1/2).
\end{array}
\end{equation}
To uniquely define the $\gamma^\pm_{j}$, we need to complement (\ref{farfield}) with some conditions at the origin. To proceed, we match the behaviour of the $\gamma^\pm_{j}$ with the one of some inner field expansion of $u^\eps$. More precisely, in a neighbourhood of the origin we look for an expansion of $u^\eps$ of the form 
\begin{equation}\label{NearField}
u^\eps(x)=W(x_1/\eps)+\dots
\end{equation}
with $W$ to determine. Inserting (\ref{NearField}) and (\ref{ExpansionTR0_2}) in (\ref{PbSpectralCell}), we find that $W$ must satisfy 
\begin{equation}\label{NearFieldPb}
\begin{array}{|rcll}
\Delta W+2\pi^2W&=&0&\mbox{ in }\Om\\
W&=&0 &\mbox{ on }\partial\Om.
\end{array}
\end{equation}
Now we come to the point where Theorem \ref{ThmNoBoundedSol} appears in the analysis. Indeed it guarantees that the only solution of (\ref{NearFieldPb}) which is bounded at infinity is the null function. Therefore we take $W\equiv0$ and impose 
\begin{equation}\label{farfield_BC}
\gamma^\pm_{j}(0)=0.
\end{equation}
Then solving the spectral problem (\ref{farfield}), (\ref{farfield_BC}), we obtain 
\begin{equation}\label{DefModel1D}
\nu=p^2\pi^2 \mbox{ for }p\in\N^\ast,\qquad  \begin{array}{|l}
\gamma^+_{j}(s)=\sin(p\pi s)\\[2pt]
\gamma^-_{j}(s)=-e^{i\eta_j}\sin(p\pi s).
\end{array}
\end{equation}
Note that $\nu$ is a triple eigenvalue (geometric multiplicity equal to three) because the problems (\ref{farfield}), (\ref{farfield_BC}) for $j=1,2,3$ are uncoupled. Additionally $\nu$ is independent of $\eta\in[0;2\pi)^3$. This latter fact is not completely satisfactory and in the sequel we wish to improve the model obtained above. Let us refine the expansion proposed in (\ref{ExpansionTR0_2}) and work with
\begin{equation}\label{ExpansionTR0_3}
\Lambda^\eps=\cfrac{2\pi^2}{\eps^2}+p^2\pi^2+\eps\tilde{\nu}+\dots,\quad u^\eps(x)=\begin{array}{|rl}
(a_1\gamma^\pm_{1}(x_1)+\eps\tilde{\gamma}^{\pm}_{1}(x_1))\,U_{\dagger}(x_2/\eps,x_3/\eps)&\mbox{ in }\om^{\eps\pm}_1\\[3pt]
(a_2\gamma^\pm_{2}(x_2)+\eps\tilde{\gamma}^{\pm}_{2}(x_2))\,U_{\dagger}(x_3/\eps,x_1/\eps)&\mbox{ in }\om^{\eps\pm}_2\\[3pt]
(a_3\gamma^\pm_{3}(x_3)+\eps\tilde{\gamma}^{\pm}_{3}(x_3))\,U_{\dagger}(x_1/\eps,x_2/\eps)&\mbox{ in }\om^{\eps\pm}_3\\[3pt]
\end{array}+\dots .
\end{equation}
Here $a:=(a_1,a_2,a_3)\in\R^3$, $\tilde{\nu}$ as well as the $\tilde{\gamma}^{\pm}_{j}$, $j=1,2,3$, are to determine. Since we are working with eigenfunctions, we can impose the normalization condition $a_1^2+a_2^2+a_3^2=1$. Inserting (\ref{ExpansionTR0_3}) into Problem (\ref{PbSpectralCell}) and extracting the terms in $\eps$, we get for $j=1,2,3$,
\begin{equation}\label{farfield_3}
\begin{array}{|rcl}
\partial^2_s\tilde{\gamma}^+_{j}+p^2\pi^2\tilde{\gamma}^+_{j}&=&-\tilde{\nu}a_j\gamma^+_j\quad\mbox{ in }(0;1/2)\\[3pt]
\partial^2_s\tilde{\gamma}^-_{j}+p^2\pi^2\tilde{\gamma}^-_{j}&=&-\tilde{\nu}a_j\gamma^-_j\quad\mbox{ in }(-1/2;0)\\[3pt]
\tilde{\gamma}^-_{j}(-1/2)&=&e^{i\eta_j}\tilde{\gamma}^+_{j}(+1/2)\\[3pt]
\partial_s\tilde{\gamma}^-_{j}(-1/2)&=&e^{i\eta_j}\partial_s\tilde{\gamma}^+_{j}(+1/2).
\end{array}
\end{equation}
To define properly the $\tilde{\gamma}^\pm_{j}$, we need to add to (\ref{farfield_3}) conditions at the origin. To identify them, again we match with the behaviour of some inner field representation of $u^\eps$. In a neighbourhood of the origin we look for an expansion of $u^\eps$ of the form 
\begin{equation}\label{NearField_3}
u^\eps(x)=\eps \tilde{W}(x_1/\eps)+\dots.
\end{equation}
Inserting (\ref{NearField_3}) and (\ref{ExpansionTR0_3}) in (\ref{PbSpectralCell}), we find that $\tilde{W}$ must satisfy 
\begin{equation}\label{NearFieldPb_3}
\begin{array}{|rcll}
\Delta \tilde{W}+2\pi^2\tilde{W}&=&0&\mbox{ in }\Om\\
\tilde{W}&=&0 &\mbox{ on }\partial\Om.
\end{array}
\end{equation}
This problem admits solutions $\tilde{W}_{j}$, $j=1,\dots,6$, with the expansions
\begin{equation}\label{DefPolMat1}
\tilde{W}_j(x)=
\begin{array}{|rl}
(z_j+M_{jj})\,U_{\dagger}(y_j)+\dots&\mbox{ in }\mathcal{L}_j\mbox{ as }z_j\to+\infty \\[2pt]
M_{jk}\,U_{\dagger}(y_k)+\dots&\mbox{ in }\mathcal{L}_k\mbox{ as }z_k\to+\infty,\quad k\ne j.
\end{array}
\end{equation}
Here we use the coordinates $(z_j,y_j)$ introduced in (\ref{Def_coor_local}). Note that at infinity $\tilde{W}_j$ is growing only in the branch $\mathcal{L}_j$. Let us explain how to show the existence of these functions (see \cite[Chap. 4, Prop. 4.13]{NaPl94} for more details). Introduce some $\psi\in\mathscr{C}^{\infty}(\R)$ such that $\psi(s)=1$ for $s>1$ and $\psi(s)=0$ for $s\le 1/2$. Then for $j=1,\dots,6$, define $f_j$ such that $f_j(x)=(\Delta+2\pi^2)(\psi(z_j) \,z_jU_{\dagger}(y_j))$. Observe that $f_j$ is compactly supported. Then the theory of \cite[Chap. 5]{NaPl94} together with Theorem \ref{ThmNoBoundedSol} above ensure that there is a solution to the problem
\[
\begin{array}{|rcll}
\Delta \hat{W}_j+2\pi^2\hat{W}_j&=&-f_j&\mbox{ in }\Om\\
\hat{W}_j&=&0 &\mbox{ on }\partial\Om
\end{array}
\]
which is bounded at infinity. Finally we take $\tilde{W}_j$ such that $\tilde{W}_j(x)=\psi(z_j) \,z_jU_{\dagger}(y_j)+\hat{W}_j(x)$. The coefficients $M_{jk}$ form the so-called polarization matrix
\begin{equation}\label{def_Pola_matrix}
\mathbb{M}:=\left(M_{jk}\right)_{1\le j,k\le 6}
\end{equation}
which is real and symmetric even in non symmetric geometries (see \cite[Chap. 5, Prop. 4.13]{NaPl94}).\\
\newline
For the functions $\gamma^{\pm}_{j}$ in (\ref{DefModel1D}), we have the Taylor expansion, as $s\to0$,
\begin{equation}\label{ExpanFarField}
\gamma^+_{j}(s)=0+p\pi s+\dots=\eps p\pi\,\cfrac{s}{\eps}+\dots,\qquad \gamma^-_{j}(s)=0-e^{i\eta_j}p\pi s+\dots=-\eps e^{i\eta_j} p\pi \,\cfrac{s}{\eps}+\dots.
\end{equation}
Comparing (\ref{ExpanFarField}) with (\ref{DefPolMat1}) leads us to choose $\tilde{W}$ in the expansion $u^\eps(x)=\eps \tilde{W}(x_1/\eps)+\dots$ (see (\ref{NearField_3})) such that
\[
\tilde{W}=p\pi \sum_{j=1}^3a_j\,(\tilde{W}_j+e^{i\eta_j}\tilde{W}_{3+j}).
\]
This sets the constant behaviour of $\tilde{W}$ at infinity and we now match the later with the behaviour of the $\tilde{\gamma}^\pm_{j}$ at the origin to close system (\ref{farfield_3}). This step leads us to impose 
\begin{equation}\label{ConditionAtZero}
\begin{array}{ll}
\dsp\tilde{\gamma}^+_1(0)=p\pi \sum_{j=1}^3 a_j\,(M_{j1}+e^{i\eta_j}M_{3+j1}); &\quad\dsp \tilde{\gamma}^-_1(0)=p\pi \sum_{j=1}^3 a_j\,M_{j4}+e^{i\eta_j}M_{3+j4});\\[14pt]
\dsp\tilde{\gamma}^+_2(0)=p\pi \sum_{j=1}^3 a_j\,(M_{j2}+e^{i\eta_j}M_{3+j2}); &\quad\dsp \tilde{\gamma}^-_2(0)=p\pi \sum_{j=1}^3 a_j\,(M_{j5}+e^{i\eta_j}M_{3+j5});\\[14pt]
\dsp\tilde{\gamma}^+_3(0)=p\pi \sum_{j=1}^3 a_j\,(M_{j3}+e^{i\eta_j}M_{3+j3}); &\quad\dsp \tilde{\gamma}^-_3(0)=p\pi \sum_{j=1}^3 a_j\,(M_{j6}+e^{i\eta_j}M_{3+j6}).
\end{array}
\end{equation}
Equations (\ref{farfield_3}), (\ref{ConditionAtZero}) form a boundary value problem for a system of ordinary differential equations. For this problem, there is a kernel and co-cokernel. In order to have a solution, the following compatibility conditions, obtained by multiplying (\ref{farfield_3}) by $\overline{\gamma^\pm_j}$ and integrating by parts,
\[
\overline{\partial_s\gamma^+_j(0)}\,\tilde{\gamma}^+_j(0)-\overline{\partial_s\gamma^-_j(0)}\,\tilde{\gamma}^-_j(0)=-\tilde{\nu}a_j\left(\int_{-1/2}^{0}|\gamma^-_j|^2\,ds+\int_{0}^{1/2}|\gamma^+_j|^2\,ds\right)
\]
must be satisfied for $j=1,2,3$ (note that we used that $\gamma^{\pm}_j(0)=0$ according to (\ref{farfield_BC})). Since $\overline{\partial_s\gamma^+_j(0)}=p\pi$ and $\overline{\partial_s\gamma^-_j(0)}=-p\pi e^{-i\eta_j}$, this gives
\[
2(p\pi)^2\bigg(\sum_{k=1}^3 a_k\,(M_{kj}+e^{i\eta_k}M_{3+kj})+e^{-i\eta_j}\sum_{k=1}^3 a_k\,(M_{k3+j}+e^{i\eta_k}M_{3+k3+j})\bigg)=-\tilde{\nu}a_j.
\]
In a more compact form and by rewriting the dependence with respect to $\eta$, we obtain
\begin{equation}\label{CompactForm}
\mathbb{A}(\eta) a^{\top}=-\tilde{\nu}(\eta)a^{\top},
\end{equation}
with 
\[
\mathbb{A}(\eta):=\Theta(\eta)\mathbb{M}\Theta^\ast(\eta)\in\Cplx^{3\times 3},\quad
\Theta(\eta):=\left(\begin{array}{cccccc}
1 & 0 & 0 & e^{-i\eta_1} & 0 & 0\\[2pt]
0& 1 & 0 & 0 & e^{-i\eta_2} & 0\\[2pt]
0 & 0 & 1 & 0 & 0 & e^{-i\eta_3} 
\end{array}\right),\quad\Theta^{\ast}(\eta)=\overline{\Theta(\eta)}^{\top}.
\]
Above we used that the polarization matrix $\mathbb{M}$ defined in (\ref{def_Pola_matrix}) is symmetric. By solving the spectral problem (\ref{CompactForm}), we get the values for $\tilde{\nu}(\eta)$ and $a$. Once $\tilde{\nu}(\eta)$ and $a$ are known, one can compute the solution to the system (\ref{farfield_3}), (\ref{ConditionAtZero}) to obtain the expressions of the $\tilde{\gamma}_j(\eta)$. This ends the definition of the terms appearing in the expansions (\ref{ExpansionTR0_3}). Let us exploit and comment these results.\\ 
\newline
$\star$ First, observe that the $\tilde{\nu}(\eta)$ are real. Indeed, since $\mathbb{M}$ is real and symmetric, we infer that $\mathbb{A}(\eta)$ is hermitian. \\
\newline
$\star$ We have obtained 
\begin{equation}\label{ExpanConclu}
\Lambda^\eps(\eta)=\eps^{-2}2\pi^2+p^2\pi^2+\eps\tilde{\nu}(\eta)+\dots\,.
\end{equation}
Note that in this expansion, the third term, contrary to the first two ones, depends on $\eta$. For $\eta\in[0;2\pi)^3$, denote by $\tilde{\nu}_1(\eta)$, $\tilde{\nu}_2(\eta)$, $\tilde{\nu}_3(\eta)$ the three eigenvalues of 
(\ref{CompactForm}) numbered such that
\[
\tilde{\nu}_1(\eta) \le \tilde{\nu}_2(\eta) \le \tilde{\nu}_3(\eta).
\]
Define the quantity 
\begin{equation}\label{defAleph}
\aleph:=\bigcup_{j=1}^3\{\tilde{\nu}_j(\eta),\,\eta\in[0;2\pi)^3\}.
\end{equation}
Since $\Upsilon^\eps_{1+q+3(p-1)}=\{\Lambda_{1+q+3(p-1)}^\eps(\eta),\,\eta\in[0;2\pi)^3\}$, this analysis shows that we have the inclusion
\begin{equation}\label{MainInclusion}
\Upsilon^\eps_{1+q+3(p-1)}\subset \eps^{-2}2\pi^2+p^2\pi^2+\eps\,(-c_p;c_p)
\end{equation}
for $\eps$ small enough as stated in Theorem \ref{MainThmPerio}. For $c_p$, we can take any value such that $\overline{\aleph}\subset (-c_p;c_p)$. \\
\newline
$\star$ Due to the symmetries of $\Om$, $\mathbb{M}$ is of the form
\begin{equation}\label{MatPolSimple}
\left(\begin{array}{cccccc}
r_m & t^{\perp}_m & t^{\perp}_m & t_m & t^{\perp}_m & t^{\perp}_m \\[2pt]
t^{\perp}_m & r_m & t^{\perp}_m & t^{\perp}_m & t_m & t^{\perp}_m \\[2pt]
t^{\perp}_m & t^{\perp}_m & r_m & t^{\perp}_m & t^{\perp}_m & t_m \\[2pt]
t_m & t^{\perp}_m & t^{\perp}_m & r_m & t^{\perp}_m & t^{\perp}_m  \\[2pt]
t^{\perp}_m & t_m & t^{\perp}_m & t^{\perp}_m & r_m & t^{\perp}_m   \\[2pt]
t^{\perp}_m & t^{\perp}_m & t_m & t^{\perp}_m & t^{\perp}_m & r_m  
\end{array}\right)
\end{equation}
where $r_m$, $t_m$, $t_m^\perp$ are real coefficients. For $\eta=(0,0,0)$, we find that the eigenvalues of (\ref{CompactForm}) are 
\[
2r_m+2t_m+8t^{\perp}_m,\qquad 2r_m+2t_m-4t^{\perp}_m,\qquad 2r_m+2t_m-4t^{\perp}_m.
\] 
For $\eta=(\pi,\pi,\pi)$, we obtain 
\[
2r_m-2t_m,\qquad 2r_m-2t_m,\qquad  2r_m-2t_m.
\]
In the numerics of \S\ref{Section_PolMat}, we compute $\mathbb{M}$. From the values of $r_m$, $t_m$, $t^{\perp}_m$ obtained in (\ref{ValuePolMat}), we find for example
\[
2r_m+2t_m+8t^{\perp}_m \ne 2r_m-2t_m.
\]
This shows that the set $\aleph$ appearing in (\ref{defAleph}) has a non empty interior. Additionally, since $\nu_1(\pi)=\nu_2(\pi)=\nu_3(\pi)$ and the $\nu_j$ depend continuously on $\eta$, we deduce that $\aleph$ is a connected segment. This is obtained in the numerics of \S\ref{Section_PolMat} (see (\ref{NumAleph})) and due to the symmetries of $\om^\eps$, this was somehow expected.\\
\newline
Finally, let us mention again that all the formal presentation above can be justified rigorously by a direct adaptation of the proofs of error estimates presented in \cite{Naza17,Naza18}.

\section{Numerics}\label{SectionNumerics}

\subsection{Spectrum of $A^{\Om}$}

We start the numerics by computing the spectrum of the operator $A^{\Om}$ defined after (\ref{PbSpectralZoom}). More precisely, in order to approximate also the eigenvalues which are embedded in the continuous spectrum of $A^{\Om}$ and to reveal complex resonances which would be located close to the real axis, we work with Perfectly Matched Layers \cite{Bera94,KiPa10,BoCP18} (see also the techniques of analytic dilatations \cite{aguilar1971class,balslev1971spectral,moiseyev1998quantum}). For $\theta\in(0;\pi/2)$ and $L>1/2$, define the complex valued parameters 
\[
\alpha^\theta_1=\begin {array}{|ll}
1 & \mbox{for }|x_1|\le L\\[2pt]
e^{-i\theta} & \mbox{for }|x_1|>L
\end{array},\qquad \alpha^\theta_2=\begin {array}{|ll}
1 & \mbox{for }|x_2|\le L\\[2pt]
e^{-i\theta} & \mbox{for }|x_2|>L
\end{array},\qquad \alpha^\theta_3=\begin {array}{|ll}
1 & \mbox{for }|x_3|\le L\\[2pt]
e^{-i\theta} & \mbox{for }|x_3|>L.
\end{array}
\]
Here the coefficient $\theta$ will drive the rotation of the continuous spectrum while $L$ marks the beginning of the PML region. Then consider the spectral problem
\begin{equation}\label{ProblemPMLs}
\begin{array}{|rcll}
-\alpha^\theta_1\cfrac{\partial}{\partial x_1}\Big(\alpha^\theta_1\cfrac{\partial u}{\partial x_1}\,\Big)+\alpha^\theta_2\cfrac{\partial}{\partial x_2}\Big(\alpha^\theta_2\cfrac{\partial u}{\partial x_2}\,\Big)+\alpha^\theta_3\cfrac{\partial}{\partial x_3}\Big(\alpha^\theta_3\cfrac{\partial u}{\partial x_3}\,\Big)&=&\Lambda u &\mbox{ in } \Om\\[3pt]
u&=&0&\mbox{ on } \partial\Om.
\end{array}
\end{equation}
Denote by $A^\Om_\theta$ the unbounded operator associated with (\ref{ProblemPMLs}). Observe that $A^\Om_\theta$ is not selfadjoint due to the complex parameters $\alpha^\theta_j$. However the theory guarantees that the real eigenvalues of $A^\Om_\theta$ coincide exactly with the real eigenvalues of $A^{\Om}$. What is interesting is that one can show that the essential spectrum of $A^\Om_\theta$, that is the values of $\Lambda$ such that $A^\Om_\theta-\Lambda\,\mrm{Id}:\mathcal{D}(A^\Om_\theta)\to\mL^2(\Om)$ is not Fredholm, corresponds to the set
\[
\bigcup_{m,n\in\N^\ast}\{\pi^2(m^2+n^2)+t\,e^{-2i\theta},\,t\ge0\},
\]
so that the real eigenvalues of $A^\Om_\theta$ are isolated in the spectrum. As a consequence, we can compute them by truncating the branches of $\Om$ at a certain distance without producing spectral pollution. Then we approximate the spectrum in this bounded geometry by using a classical P1 finite element method. We construct the matrices with the library \texttt{Freefem++} \cite{Hech12} and compute the spectrum with \texttt{Matlab}\footnote{\textit{Matlab}, \url{http://www.mathworks.com/}.}.\\
\newline
In Figure \ref{FigSpectrum}, we display in the complex plane the approximation of the spectrum of $A^\Om_\theta$ obtained with this approach. We observe that the branches of essential spectrum of $A^\Om_\theta$ are somehow discretized. This is due to the fact that the approximated problem is set in finite dimension. Moreover, we note that $A^\Om_\theta$, and so $A^\Om$, have eigenvalues on the real line. In accordance with Theorem \ref{ThmDiscreteSpectrum}, we find exactly one eigenvalue $\mu_1\approx 12.9\approx1.3\pi^2$ on the segment $(0;\pi^2)$. We also note the presence of eigenvalues embedded in the continuous spectrum for the operator $A^{\Om}$. For the first one, we get $\mu_2\approx 46.7\approx4.7\pi^2$. 

\begin{figure}[!ht]
\centering
\includegraphics[width=10cm]{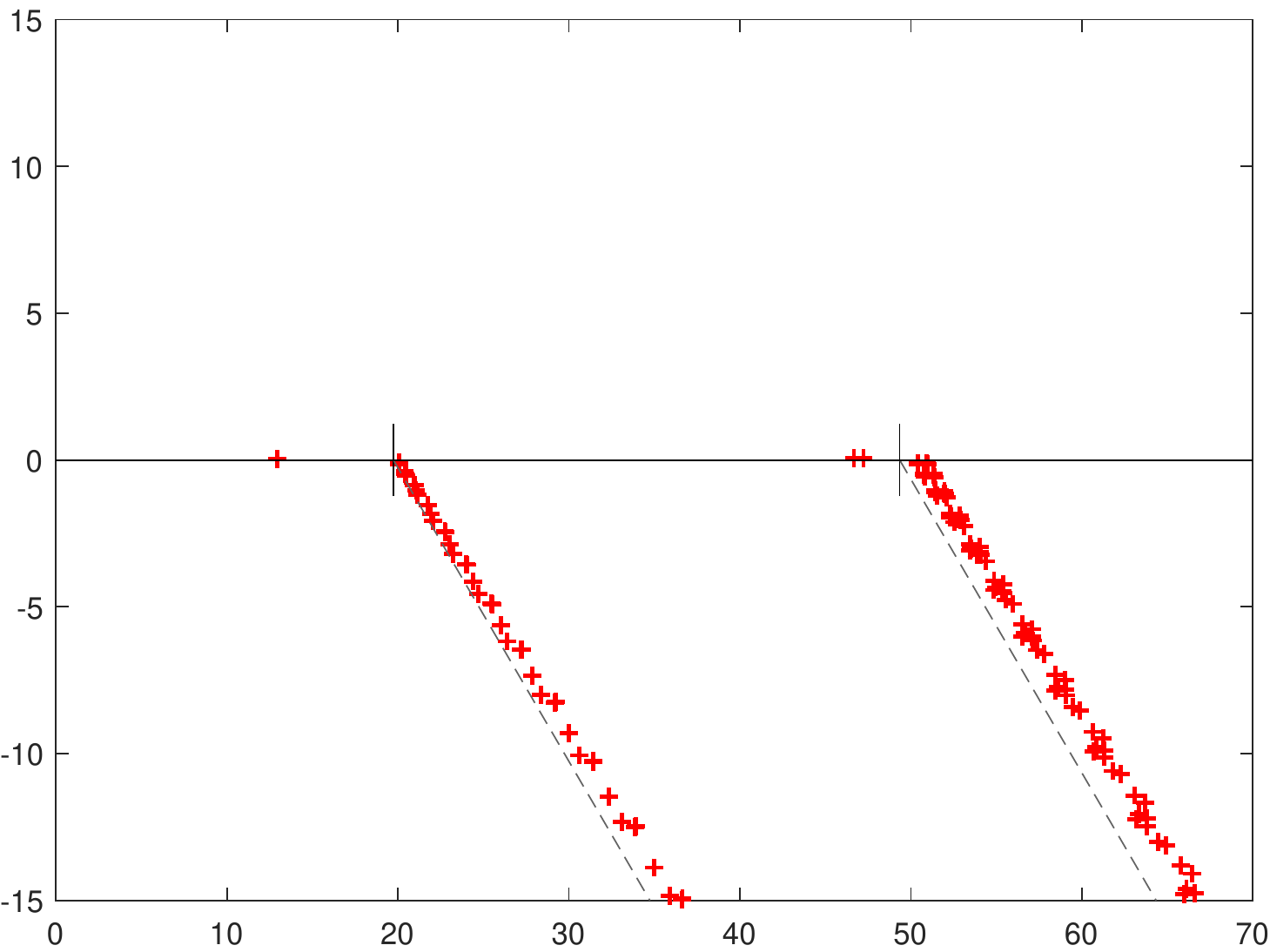}
\caption{Approximation of the spectrum of $A^\Om_\theta$ for $\theta=\pi/4$ ($L=1/2$).}
\label{FigSpectrum}
\end{figure}

\noindent In Figure \ref{FigEigenfunctions}, we represent an eigenfunction (trapped mode) associated with the eigenvalue $\mu_1$ of the discrete spectrum of $A^{\Om}$. As guaranteed by Lemma \ref{LemmaSym}, we observe that it is indeed symmetric with respect to the planes $(Ox_2x_3)$, $(Ox_3x_1)$, $(Ox_1x_2)$.

\begin{figure}[!ht]
\centering
\includegraphics[width=5.8cm]{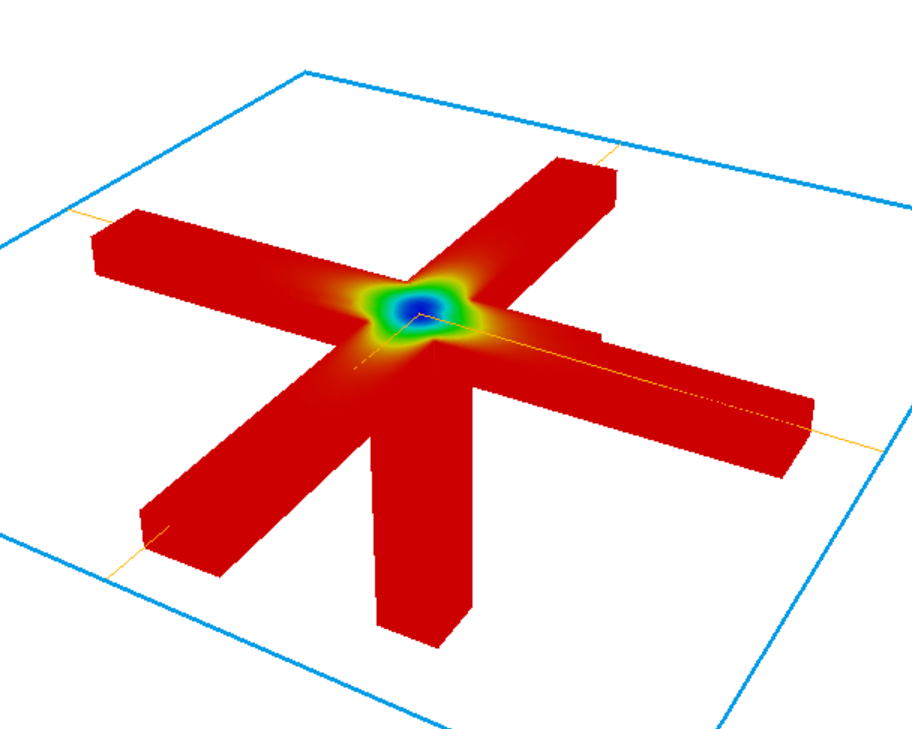}\qquad\quad\includegraphics[width=5.8cm]{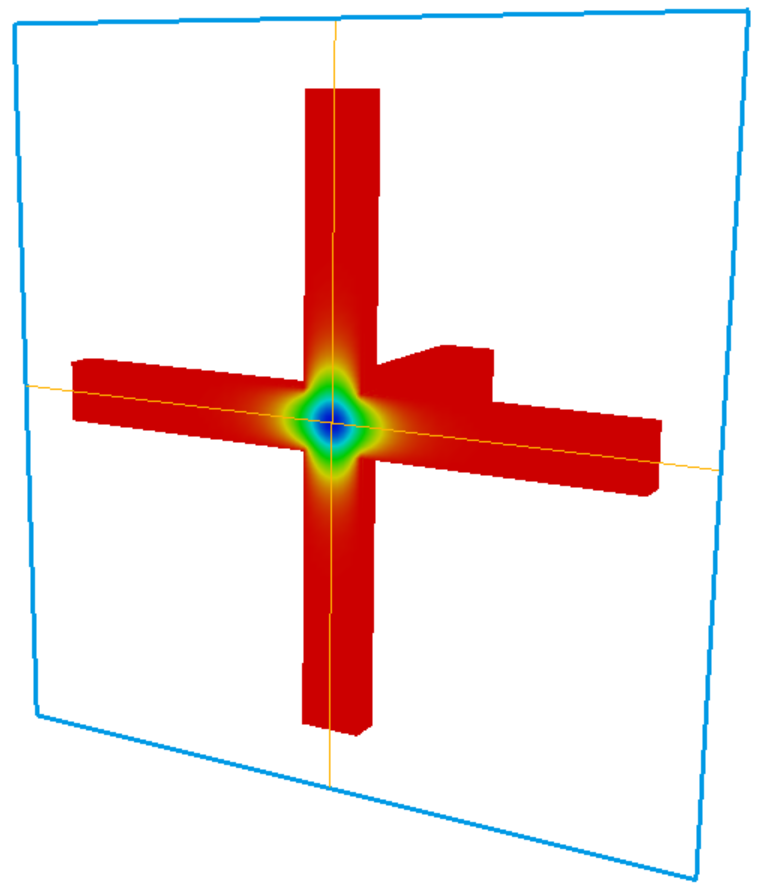}
\caption{Eigenfunction associated with the eigenvalue $\mu_1$: cuts $x_1=0$ (left) and $x_2=0$ (right).}
\label{FigEigenfunctions}
\end{figure}

\subsection{Polarization matrix and threshold scattering matrix for the operator $A^\Om$}\label{Section_PolMat}
In this section, we explain how to compute the polarization matrix $\mathbb{M}$ introduced after (\ref{CompactForm}) and whose properties allow one to assess the second corrector term in the expansion $\Lambda^\eps=\eps^{-2}2\pi^2+p^2\pi^2+\eps\tilde{\nu}+\dots$ of the eigenvalues generating the spectral bands $\Upsilon^\eps_k$, $k\ge 2$. We work in the geometry $\Om$ defined in (\ref{DefOmega}) and study the problem (\ref{PbSpectralZoom}) at the threshold, namely 
\begin{equation}\label{PbSpiderThreshold}
\begin{array}{|rcll}
\Delta v +2\pi^2v&=& 0 &\mbox{on }\Om\\
v&=&0 &\mbox{on }\partial\Om.
\end{array}
\end{equation}
To obtain $\mathbb{M}$, we will first compute the so-called threshold scattering matrix that we define now. In $\mathcal{L}_j$, $j=1,\dots,6$, set 
\[
w_{j}^{\pm}(x)=\cfrac{z_j \mp i}{\sqrt{2}}\,U_{\dagger}(y_j).
\]
Let us work again with the function $\psi\in\mathscr{C}^{\infty}(\R)$ introduced after (\ref{DefPolMat1}) such that $\psi(s)=1$ for $s>1$ and $\psi(s)=0$ for $s\le 1/2$. For $j=1,\dots,6$, define $\psi_j$ such that $\psi_j(x)=\psi(z_j,y_j)$ (observe that $\psi_j$ is non zero only in the branch $\mathcal{L}_j$). For $j=1,\dots,6$, the theory of \cite[Chap. 5]{NaPl94} guarantees that problem (\ref{PbSpiderThreshold}) admits a solution with the decomposition
\begin{equation}\label{decompo_vj}
v_j=\psi_j\,w_j^{-}+\sum_{k=1}^6\psi_k\,s_{jk}\,w_k^++\tilde{u}_j,
\end{equation}
where the $s_{jk}$ are complex numbers and the $\tilde{u}_j$ decay exponentially at infinity. The matrix 
\[
\mathbb{S}:=\left(s_{jk}\right)_{1\le j,k\le 6}
\]
is called the threshold scattering matrix. It is symmetric ($\mathbb{S}=\mathbb{S}^{\top}$) but not necessarily hermitian and unitary ($\mathbb{S}\,\overline{\mathbb{S}}^{\top}=\mrm{Id}$). It is known (see relation (7.9) in \cite{Naza16}) that $\mathbb{M}$ coincides with the Cayley transform of $\mathbb{S}$, i.e. we have
\begin{equation}\label{DefMatPol}
\mathbb{M}=i(\mrm{Id}+\mathbb{S})^{-1}(\mrm{Id}-\mathbb{S}).
\end{equation}
Note that one can show that we have $\dim\,\ker\,(\mrm{Id}+\mathbb{S})=\dim\,(\mathscr{B}/\mathscr{B}_{\mrm{tr}})$ (the quotient space) where $\mathscr{B}$ denotes the space of bounded solutions of (\ref{PbSpiderThreshold}) and $\mathscr{B}_{\mrm{tr}}$ the space of trapped modes of (\ref{PbSpiderThreshold}). For the proof, we refer the reader for example to Theorem 1 in \cite{Naza14}. Since Theorem \ref{ThmNoBoundedSol} ensures that $\mathscr{B}$ reduces to the null function, we infer that $\mrm{Id}+\mathbb{S}$ is invertible which guarantees that $\mathbb{M}$ is well defined via formula (\ref{DefMatPol}). Additionally, due to the symmetries of $\Om$, $\mathbb{S}$ is of the form
\begin{equation}\label{DefMatScaSimplifiee}
\left(\begin{array}{cccccc}
r & t^{\perp} & t^{\perp} & t & t^{\perp} & t^{\perp} \\[2pt]
t^{\perp} & r & t^{\perp} & t^{\perp} & t & t^{\perp} \\[2pt]
t^{\perp} & t^{\perp} & r & t^{\perp} & t^{\perp} & t \\[2pt]
t & t^{\perp} & t^{\perp} & r & t^{\perp} & t^{\perp}  \\[2pt]
t^{\perp} & t & t^{\perp} & t^{\perp} & r & t^{\perp}   \\[2pt]
t^{\perp} & t^{\perp} & t & t^{\perp} & t^{\perp} & r  
\end{array}\right)
\end{equation}
where $r$, $t$, $t^{\perp}$ are complex reflection and transmission coefficients. Therefore it is sufficient to compute $v_1$. To proceed, we shall work in the bounded domain $\Om^R$ (see before (\ref{SpectralRMixed})) and impose approximated radiation conditions on the artificial cuts. Denote by $n$ the unit normal to $\partial\Om^R$ directed to the exterior of $\Om^R$, set 
\[
\begin{array}{l}
\Gamma_1:=\{R\}\times(-1/2;1/2)^2,\quad  \Gamma_2:=(-1/2;1/2)\times\{R\}\times(-1/2;1/2),\quad  \Gamma_3:=(-1/2;1/2)^2\times\{R\}\\[4pt]
\Gamma_4:=\{-R\}\times(-1/2;1/2)^2,\  \Gamma_5:=(-1/2;1/2)\times\{-R\}\times(-1/2;1/2),\  \Gamma_6:=(-1/2;1/2)^2\times\{-R\}
\end{array}
\]
and $\Gamma:=\cup_{j=1}^6\Gamma_j$. On $\Gamma_1$, according to (\ref{decompo_vj}), we have
\[
\partial_n(v_1-\psi_1w_{1}^{-})=\partial_{z_1}(v_1-w_{1}^{-})=2^{-1/2}r\,U_{\dagger}+\dots
\]
where the dots stand for terms which are small for large values of $R$. On the other hand on $\Gamma_1$, there holds
\[
v_1-w_{1}^{-}=2^{-1/2}r\,(R-i)\,U_{\dagger}+\dots .
\]
Therefore, this gives, still on $\Gamma_1$, 
\begin{equation}\label{Approx_RC1}
\partial_n v_1=\cfrac{v_1-w_{1}^{-}}{R-i}+\partial_{z_1} w_{1}^{-}+\dots=\cfrac{v_1}{R-i}+w_{1}^{-}\left(\cfrac{1}{R+i}-\cfrac{1}{R-i}\right)+\dots=\cfrac{v_1}{R-i}-\cfrac{2i}{R^2+1}\,w_{1}^{-}+\dots.
\end{equation}
On $\Gamma_j$, $j\ne1$, the situation is simpler because $v_1$ is outgoing in the corresponding branches and we have
\begin{equation}\label{Approx_RC2}
\partial_n v_1=\cfrac{v_1}{R-i}+\dots .
\end{equation}
Finally, using the Robin conditions (\ref{Approx_RC1}), (\ref{Approx_RC2}) as approximated radiation conditions, we consider the variational formulation
\begin{equation}\label{Pb_Approximation}
\begin{array}{|l}
\mbox{Find }\hat{v}_1\in\mH^1_{0}(\Om^R;\partial\Om^R\cap\partial\Om)\mbox{ such that for all }v\in\mH^1_{0}(\Om^R;\partial\Om^R\cap\partial\Om) \\[6pt]
\dsp\int_{\Om^R}\nabla \hat{v}_1\cdot\nabla v\,dx-\cfrac{1}{R-i}\int_{\Gamma} \hat{v}_1 v\,d\sigma-(2\pi)^2\dsp\int_{\Om^R}\hat{v}_1 v\,dx=-\cfrac{2i}{R^2+1}\int_{\Gamma_1} w_{1}^{-} v\,d\sigma.
\end{array}
\end{equation}
One can prove that $\hat{v}_1$ yields a good approximation of $v_1$ with an error which is exponentially decaying with $R$. In practice, we solve the problem (\ref{Pb_Approximation}) with a P1 finite element method thanks to \texttt{Freefem++}. Then replacing $v_1$ by $\hat{v}_1$ in the exact formulas
\begin{equation}\label{FormulaApproxCoef}
r=\cfrac{2}{R^2+1}\,\int_{\Gamma_1} (v_1-w_{1}^{-})\,w_{1}^{-}\,d\sigma,\quad t=\cfrac{2}{R^2+1}\,\int_{\Gamma_4} v_1\,w_4^{-}\,d\sigma,\quad t^\perp=\cfrac{2}{R^2+1}\,\int_{\Gamma_2} v_1\,w_2^{-}\,d\sigma,
\end{equation}
we get an approximation of the threshold scattering matrix $\mathbb{S}$ given by (\ref{DefMatScaSimplifiee}). Finally with (\ref{DefMatPol}), we obtain an approximation of the polarization matrix $\mathbb{M}$ which appears in the $3\times3$ spectral problems (\ref{CompactForm}).\\
\newline
Our computations give
\[
r\approx 0.66+0.11i,\qquad t\approx 0.08-0.70i,\qquad t^{\perp}\approx -0.08-0.08i.
\]
The eigenvalues of $\mathbb{S}\in\mathbb{C}^{6\times6}$ are approximately equal to 
\[
0.44 - 0.9i\mbox{ (simple)},\quad 0.9 - 0.44i\mbox{ (double)}\quad\mbox{ and }\quad 0.58 + 0.81i\mbox{ (triple).}
\]
They have modulus one which is consistent with the fact that $\mathbb{S}$ is unitary. Moreover, we indeed observe that they are different from $-1$ which is coherent with the discussion following (\ref{DefMatPol}) (absence of threshold resonance). On the other hand, for the coefficients of $\mathbb{M}$ (see (\ref{MatPolSimple})), we get
\begin{equation}\label{ValuePolMat}
r_m\approx 0.08,\qquad t_m\approx -0.44,\qquad t^{\perp}_m\approx -0.06.
\end{equation}
With these values, solving the $3\times3$ eigenvalue problem (\ref{CompactForm}) for $\eta\in[0;2\pi)^3$, we find that the segment $\aleph$ defined in (\ref{defAleph}) satisfies
\begin{equation}\label{NumAleph}
\aleph\approx (-1.24;1.04).
\end{equation}

\noindent Let us mention that the Robin conditions (\ref{Approx_RC1}), (\ref{Approx_RC2}) are rather crude approximations of the exact radiation conditions. To get good errors estimates, we should take rather large values of $R$. However in practice, large $R$ are no so simple to handle and can create important numerical errors. Therefore a compromise must be found and we take $R=2.5$. Admittedly, this point should more investigated.

\section*{Apprendix}

\subsection{Friedrichs inequality}
We reproduce here the Lemma 5.1 of \cite{BaMaNa17}.
\begin{lemma}\label{LemmaPoincareFriedrichs}
Assume that $a>0$. Then we have the Friedrichs inequality 
\begin{equation}\label{Fried_ineq}
 \kappa(a)\int_{0}^{1/2}\phi^2\,dt \le \int_{0}^{+\infty}(\partial_t\phi)^2\,dt+a^2\int_{1/2}^{+\infty}\phi^2\,dt,\qquad \forall \phi\in\mH^1(0;+\infty),
\end{equation}
where $\kappa(a)$ is the smallest positive root of the transcendental equation
\begin{equation}\label{TransEqnLemma}
\sqrt{\kappa}\tan\bigg(\cfrac{\sqrt{\kappa}}{2}\bigg)=a.
\end{equation}
\end{lemma}
\begin{proof}
For $a>0$, consider the spectral problem 
\begin{equation}\label{Eigen1D}
\begin{array}{|rcll}
-\partial^2_t \phi+a^2\mathbbm{1}_{(1/2;+\infty)} \phi &=& \lambda(a)\,\mathbbm{1}_{(0;1/2)}\phi&\mbox{ in }(0;+\infty)\\[4pt]
\partial_t\phi(0)&=&0
\end{array}
\end{equation}
where $\mathbbm{1}_{(1/2;+\infty)}$, $\mathbbm{1}_{(0;1/2)}$ stand for the indicator functions of the sets $(1/2;+\infty)$, $(0;1/2)$ respectively. Let us equip $\mH^1(0;+\infty)$ with the inner product
\[
(\phi,\phi')_a=\int_{0}^{+\infty}\partial_t\phi\,\partial_t\phi'\,dt+a^2\phi\,\phi'\,dt.
\]
With the Riesz representation theorem, define the linear and continuous operator $T:\mH^1(0;+\infty)\to \mH^1(0;+\infty)$ such that
\[
(T\phi,\phi')_a=\int_{0}^{1/2}\phi\,\phi'\,dt.
\]
With this definition, we find that $(\lambda(a),\phi)$ is an eigenpair of (\ref{Eigen1D}) if and only if we have 
\[
T\phi=(\lambda(a)+a^2)^{-1}\phi.
\]
Since $T$ is bounded and symmetric, it is self-adjoint. Additionally the Rellich theorem ensures that $T$ is compact. This guarantees that the spectrum of (\ref{Eigen1D}) coincides with a sequence of positive eigenvalues whose only accumulation point is $+\infty$. Let us denote by $\kappa(a)$ the smallest eigenvalue of (\ref{Eigen1D}). From classical results concerning compact self-adjoint (see e.g. \cite[Thm. 2.7.2]{BiSo87}), we know that
\begin{equation}\label{Def_quotient}
(\kappa(a)+a^2)^{-1}=\underset{\phi\in\mH^1(0;+\infty),\,(\phi,\phi)_a=1}{\sup} (T\phi,\phi)_a.
\end{equation}
Rearranging the terms, we find that (\ref{Def_quotient}) provides the desired estimates (\ref{Fried_ineq}). Now we compute $\kappa(a)$. Solving the ordinary differential equation (\ref{Eigen1D}) with $\lambda(a)=\kappa(a)$, we obtain, up to a multiplicative constant,
\[
\phi(t)=\begin{array}{|ll}
\cos(\sqrt{\kappa(a)}t) & \mbox{ for }t\in(0;1/2)\\[3pt]
c\,e^{-at} & \mbox{ for }t\ge 1/2
\end{array}
\]
where $c$ is a constant to determine. Writing the transmission conditions at $t=1/2$, we find that a non zero solution exists if and only if $\kappa(a)>0$ satisfies the relation (\ref{TransEqnLemma}).
\end{proof}
\begin{lemma}\label{LemmaPoincareFriedrichsR}
Assume that $a>0$ and $R>1/2$. Then we have the Friedrichs inequality 
\begin{equation}\label{Fried_ineq_R}
 \kappa(a,R)\int_{0}^{1/2}\phi^2\,dt \le \int_{0}^{R}(\partial_t\phi)^2\,dt+a^2\int_{1/2}^{R}\phi^2\,dt,\qquad \forall \phi\in\mH^1(0;R),
\end{equation}
where $\kappa(a,R)$ is the smallest positive root of the transcendental equation
\begin{equation}\label{TransEqnLemmaR}
\sqrt{\kappa}\tan\bigg(\cfrac{\sqrt{\kappa}}{2}\bigg)=a\,\tanh(a(R-1/2)).
\end{equation}
Therefore, we have $\lim_{R\to+\infty}\kappa(a,R)=\kappa(a)$ where $\kappa(a)$ is the constant appearing in Lemma \ref{LemmaPoincareFriedrichs}.
\end{lemma}
\begin{proof}
The demonstration is completely similar to the one of Lemma \ref{LemmaPoincareFriedrichs} above. We find that the largest constant $ \kappa(a,R)$ such that (\ref{Fried_ineq_R}) holds coincides with the smallest eigenvalue of the problem
\[
\begin{array}{|rcll}
-\partial^2_t \phi+a^2\mathbbm{1}_{(1/2;R)} \phi &=& \kappa(a,R)\mathbbm{1}_{(0;1/2)}\phi&\mbox{ in }(0;R)\\[4pt]
\partial_t\phi(0)=\partial_t\phi(R)&=&0.
\end{array}
\]
Solving it, we find that if $\phi$ is a corresponding eigenfunction, up to a multiplicative constant, we have
\[
\phi(t)=\begin{array}{|ll}
\cos(\sqrt{\kappa(a,R)}t) & \mbox{ for }t\in(0;1/2)\\[3pt]
c\,\cosh(a(t-R)) & \mbox{ for }t\in(1/2;R)
\end{array}
\]
for some constant $c$. This times, writing the transmission conditions at $t=1/2$, we find that a non zero solution exists when $\kappa(a,R)$ satisfies (\ref{TransEqnLemmaR}). Finally we obtain that $\kappa(a,R)\to\kappa(a)$ because $\tanh(a(R-1/2))$ in (\ref{TransEqnLemmaR}) tends to $1$ when $R\to+\infty$.
\end{proof}

\section{Acknowledgements}
The work of the second author was supported by the Russian Science Foundation, project 22-11-00046.

\bibliography{Bibli}
\bibliographystyle{plain}
\end{document}